\documentclass[12pt]{article}%
\usepackage{amsmath}
\usepackage{amssymb}
\usepackage{amsfonts}
\usepackage{hyperref}
\usepackage{tikz}
\usepackage{CJK}
\usepackage{verbatim}
\usepackage[normalem]{ulem}%
\usepackage{graphicx}
\providecommand{\U}[1]{\protect\rule{.1in}{.1in}}
\usetikzlibrary{matrix,arrows}
\numberwithin{equation}{section}
\newtheorem{theorem}{Theorem}[section]

\newtheorem{corollary}[theorem]{Corollary}

\newtheorem{definition}[theorem]{Definition}
\newtheorem{example}[theorem]{Example}

\newtheorem{lemma}[theorem]{Lemma}

\newtheorem{proposition}[theorem]{Proposition}
\newtheorem{remark}[theorem]{Remark}

\newenvironment{proof}[1][Proof]{\noindent\textbf{#1.} }{\ \rule{0.5em}{0.5em}}

\newcommand{\End}{\text{End}\,}

\usetikzlibrary{arrows}
\begin{document}

\title{Weak Hopf Algebras, Smash Products and Applications to Adjoint-Stable Algebras }
\author{Zhimin Liu \thanks{E-mail: liuzm@fudan.edu.cn} \thanks{Project funded by China
Postdoctoral Science Foundation grant 2019M661327}\hspace{1cm} Shenglin Zhu
\thanks{CONTACT: mazhusl@fudan.edu.cn}\\Fudan University, Shanghai 200433, China}
\date{}
\maketitle

\begin{abstract}
For a semisimple quasi-triangular Hopf algebra $\left(  H,R\right)  $ over a
field $k$ of characteristic zero, and a strongly separable quantum commutative
$H$-module algebra $A$ over which the Drinfeld element of $H$ acts trivially,
we show that $A\#H$ is a weak Hopf algebra, and it can be embedded into a weak
Hopf algebra $\operatorname{End}A^{\ast}\otimes H$. With these structure,
$_{A\#H}\operatorname{Mod}$ is the monoidal category introduced by Cohen and
Westreich, and $_{\operatorname{End}A^{\ast}\otimes H}\mathcal{M}$ is tensor
equivalent to $_{H}\mathcal{M}$. If $A$ is in the M{\"{u}}ger center of
$_{H}{\mathcal{M}}$, then the embedding is a quasi-triangular weak Hopf
algebra morphism. This explains the presence of a subgroup inclusion in the
characterization of irreducible Yetter-Drinfeld modules for a finite group algebra.

\end{abstract}

\textbf{KEYWORDS}: Weak Hopf Algebras, Quantum Commutative, Smash Products, Almost-Triangularity

\textbf{2020 MATHEMATICS SUBJECT CLASSIFICATION}: 16T05, 16S40

\section{Introduction}

Let $\left(  H,R\right)  $ be a semisimple quasi-triangular Hopf algebra over
an algebraically closed field $k$ of characteristic zero. In the braided
tensor category $_{H}\mathcal{M}$, Majid~\cite{Majid1991Braided} constructed a
Hopf algebra $H_{R}$, called the transmuted braided group of $H$. Simple left
comodules over $H_{R}$ are used by the authors~\cite{LiuZhu2019On} to
characterize irreducible Yetter-Drinfeld modules in ${}_{H}^{H}\mathcal{YD}$.
To accomplish that purpose, a notion of $R$-adjoint-stable algebra $N_{W}$ was
introduced for each left $H_{R}$-comodule $W$, and each irreducible
Yetter-Drinfeld module in ${}_{H}^{H}\mathcal{YD}$ can be interpreted by a
simple left $H_{R}$-comodule $W$ and a simple right $N_{W}$-module.

For a simple left $H_{R}$-comodule $W$, the $R$-adjoint-stable algebra $N_{W}
$ is always of Frobenius-type (\cite[Proposition~5.15]{LiuZhu2019On}), i.e.,
the dimension of any simple module divides the dimension of $N_{W}$. There are
some structural results on the $R$-adjoint-stable algebra in certain special
cases. If $\left(  H,R\right)  =\left(  kG,1\otimes1\right)  $ for a finite
group $G$, every simple left $H$-comodule $W$ has the form $W=kg $ for some
$g\in G$, and $N_{W}=kC(g)$ is a Hopf subalgebra of $H$, where $C(g)$ is the
centralizer subgroup of $g$ in $G $. If $\left(  H,R\right)  $ is
factorizable, then the algebra $N_{W}$ of any simple left $H_{R}$-comodule $W$
is isomorphic to $H^{op}$ (see~\cite{LiuZhu2022Factorizable}).

It is natural to ask how close is the $R$-adjoint-stable algebra of a left
$H_{R}$-comodule to a Hopf algebra.

A natural connection between an $R$-adjoint-stable algebra and $H$-smash
product can be established. Any simple left $H_{R}$-comodule $W$ is affiliated
to a minimal $H$-module subcoalgebra $D$ of $H_{R}$, and a full matrix algebra
of $N_{W}$ is isomorphic to a full matrix algebra of $N_{D}$, which is
isomorphic to $D^{\ast}\#H^{op}$ (see Proposition~\ref{prop-matrix-equivalent}%
). In~\cite{Cohen1994Supersymmetry}, Cohen and Westreich introduced a monoidal
structure on the category $_{A\#H}\operatorname{Mod}$ of left $A\#H$-modules
for a quantum commutative $H$-module algebra $A$. Their result provides other
evidence that the $R$-adjoint-stable algebras are related to weak Hopf
algebras, in the case when $H_{R}{}^{\ast}$ is quantum commutative.

The quantum commutativity of $H_{R}{}^{\ast}$ can be shown to be equivalent to
the almost triangularity of $\left(  H,R\right)  $, a concept introduced
in~\cite{LiuZhu2007Almost}. For a semisimple almost-triangular Hopf algebra
$\left(  H,R\right)  $, we show that the $R$-adjoint-stable algebra $N_{D}$
for any minimal $H$-module subcoalgebra $D$ of $H_{R}$ has a natural structure
of almost-triangular weak Hopf algebra.

We do this in a more general way by studying the weak Hopf algebra structure
on $A\#H$, where $\left(  H,R\right)  $ is a semisimple quasi-triangular Hopf
algebra and $A$ is a strongly separable quantum commutative $H$-module
algebra. If the Drinfeld element $u=S\left(  R^{2}\right)  R^{1}$ of $H$ acts
on $A$ trivially, then a natural weak Hopf algebra structure on $A\#H$ can be
constructed, which leads to the monoidal structure on $_{A\#H}%
\operatorname{Mod}$ given by Cohen and Westreich~\cite{Cohen1994Supersymmetry}%
. This construction also extends the notion of transformation groupoid. When
$\left(  H,R\right)  =\left(  kG,\ 1\otimes1\right)  $ where $G$ is a finite
group, then $A\#H$ is the groupoid algebra of a transformation groupoid. (One
may refer to~\cite{Renault1980Groupoid} for the definition of groupoid.)
Moreover, if $G$ acts on $A$ transitively then $A\#H\cong M_{\dim A}\left(
kG_{1}\right)  $, where $G_{1}$ is a suitable subgroup of $G$. So there is a
canonical weak Hopf algebra embedding $A\#kG\hookrightarrow M_{\dim A}\left(
kG\right)  $.

Similar statements can be established to Hopf algebras. We prove that if
$\left(  H,R\right)  $ is a semisimple quasi-triangular Hopf algebra and $A$
is a strongly separable quantum commutative $H$-module algebra, then
$\operatorname{End}A^{\ast}\otimes H$ has a natural structure of
quasi-triangular weak Hopf algebra. If $u$ acts on $A$ trivially, then there
is a canonical embedding $A\#H\rightarrow\operatorname{End}A^{\ast}\otimes H$
of weak Hopf algebras. If assume further that $A$ is in the M{\"{u}}ger center
of $_{H}\mathcal{M}$, then a quasi-triangular structure of $A\#H$ can be also
constructed, and the given embedding is also a quasi-triangular weak Hopf
algebra map. With the defined quasi-triangular weak Hopf algebra structure,
the braided tensor category ${}_{\operatorname{End}A^{\ast}\otimes
H}\mathcal{M}$ is braided tensor equivalent to $_{H}\mathcal{M}$. When $A$ is
$H$-simple, the Frobenius-Perron dimension is calculated for any object of
$_{A\#H}\mathcal{M}$.

If $H$ is a finite dimension Hopf algebra, then it is known that $H$ is a
natural quantum commutative $D\left(  H\right)  $-module algebra. This
statement for the case that $H$ is cocommutative Hopf algebra was firstly
proven by Cohen and Westreich~\cite{Cohen1994Supersymmetry}. We show that
$H\#D\left(  H\right)  \cong\mathcal{H}\left(  H\right)  \otimes H$ as
algebras, where $\mathcal{H}\left(  H\right)  $ is the Heisenberg double of
$H$, and establish a tensor category equivalence between $_{H\#D\left(
H\right)  }\mathcal{M}$ and $_{H}\mathcal{M}$.

For an almost-triangular Hopf algebra $\left(  H,R\right)  $ and a minimal
$H$-module subalgebra $D$ of $H_{R}$, $D^{\ast}$ is a quantum commutative
$H$-module algebra, and $N_{D}\cong D^{\ast}\#H^{op}$. If $H$ is semisimple,
then the adjoint action of the Drinfeld element $u$ on $H_{R}{}^{\ast}$ is
trivial, so the results can be applied. In this case, $N_{D}$ is also an
almost-triangular weak Hopf algebra, and therefore, the category $_{N_{D}%
}\mathcal{M}$ is an almost-symmetric fusion category. It follows that for any
irreducible Yetter-Drinfeld module $V$, $\dim V$ is divisible by the dimension
of the subcoalgebra $D_{V}$ of $H_{R}$ generated by $V$.

This paper is organized as follows.

In section 2, we recall some basic results on quasi-triangular Hopf algebras,
transmuted braided group, and Yetter-Drinfeld modules.

Section 3 is devoted to the construction of quasi-triangular weak Hopf algebra
structures on $A\#H$, where $A$ is a quantum commutative module algebra. We
embed the weak Hopf algebra $A\#H$ into $B=\operatorname{End}A^{\ast}\otimes
H$, and show that the category $_{B}\mathcal{M}$ is braided tensor equivalent
to $_{H}\mathcal{M}$.

In section 4 we first give a short exploration of the categorical counterpart
of almost-triangularity for Hopf algebras. We also apply the main results to
$R$-adjoint-stable algebras, and show that if $\left(  H,R\right)  $ is a
semisimple almost-triangular Hopf algebra, and $D$ is a minimal $H$-module
subcoalgebra of $H_{R}$, then $_{N_{D}}\mathcal{M}$ is an almost-symmetric
fusion category.

\section{Preliminaries}

\label{Sec_Preli}

Throughout this paper, $k$ will be a field of characteristic zero. All
algebras and coalgebras are over $k$. For a coalgebra $\left(  C,\Delta
,\varepsilon\right)  $ and a right $C$-comodule $\left(  M,\rho\right)  $, we
use Sweedler's notation, by writing $\Delta\left(  c\right)  =c_{\left(
1\right)  }\otimes c_{\left(  2\right)  }$ for $c\in C$ and $\rho\left(
m\right)  =m_{\left\langle 0\right\rangle }\otimes m_{\left\langle
1\right\rangle }$ for $m\in M$.

Given an algebra $A$, we denote the category of finite dimensional left
$A$-modules by $_{A}\mathcal{M}$. Let $\rightharpoonup$ and $\leftharpoonup$
be the left and right $A$-actions on $A^{\ast}$ via the transpose of
multiplications on $A$, that is, for $a,b\in A$, $a^{\ast}\in A^{\ast}$,
\[
\left\langle a\rightharpoonup a^{\ast},b\right\rangle =\left\langle a^{\ast
},ba\right\rangle ,\quad\left\langle a^{\ast}\leftharpoonup a,b\right\rangle
=\left\langle a^{\ast},ab\right\rangle .
\]

Let $H$ be a Hopf algebra. An $H$-(co)module algebra is called $H$-simple if
it contains no non-trivial $H$-(co)stable ideal. We recall the following lemma
obtained from the proof of ~\cite[Theorem 3.5]{skryabin2007projectivity}.

\begin{lemma}
[{\cite[Theorem 3.5]{skryabin2007projectivity}}]\label{lemma_skryabin}Let $H$
be a Hopf algebra and $A$ be a finite-dimensional semisimple right $H$-simple
comodule algebra. If $V\in\mathcal{M}_{A}^{H}$ or $V\in{}_{A}\mathcal{M}^{H}$
with $1\leq\dim V<\infty$, then there exists some positive integer $l$ such
that $V^{l}$ is a free $A$-module.
\end{lemma}

A quasi-triangular Hopf algebra is a pair $\left(  H,R\right)  $, where $H$ is
a Hopf algebra and $R=R^{1}\otimes R^{2}\in H\otimes H$ is an invertible
element such that
\begin{align*}
&  R^{1}h_{(1)}\otimes R^{2}h_{(2)}=h_{(2)}R^{1}\otimes h_{(1)}R^{2}%
,~\forall~h\in H,\\
&  (\Delta\otimes id_{H})(R)={R_{1}}^{1}\otimes{R_{2}}^{1}\otimes{R_{1}}%
^{2}{R_{2}}^{2},\\
&  (id_{H}\otimes\Delta)(R)={R_{1}}^{1}{R_{2}}^{1}\otimes{R_{2}}^{2}%
\otimes{R_{1}}^{2},
\end{align*}
where and hereafter $R_{i}=R={R_{i}}^{1}\otimes{R_{i}}^{2}=\ (i=1,2,\ldots) $.

Over a quasi-triangular Hopf algebra $\left(  H,R\right)  $, a left $H$-module
algebra $A$ is called quantum commutative if $ab=\left(  R^{2}\cdot b\right)
\left(  R^{1}\cdot a\right)  $, for all $a,b\in A$.

Let $\left(  H,R\right)  $ be a quasi-triangular Hopf algebra. In
\cite{Majid1991Braided}, Majid introduced the transmuted braided group $H_{R}$
of $H$, which is a Hopf algebra in the braided tensor category $_{H}%
\mathcal{M}$. Specifically, $H_{R}$ is the algebra $H$ with the left adjoint
action $\cdot_{ad}$, the coproduct $\Delta_{R}$, and the antipode $S_{R}$
given respectively by
\begin{align*}
h\cdot_{ad}h^{\prime} &  =h_{\left(  1\right)  }h^{\prime}S\left(  h_{\left(
2\right)  }\right)  ,\ \Delta_{R}(h)=h_{(1)}S\left(  R^{2}\right)  \otimes
R^{1}\cdot_{ad}h_{(2)},\\
S_{R}\left(  h\right)   &  =R^{2}S\left(  R^{1}\cdot_{ad}h\right)  \text{,
where }h,\ h^{\prime}\in H\text{. }%
\end{align*}
To avoid confusion, we write $\Delta_{R}(h)=h^{(1)}\otimes h^{(2)}$.

It is known from~\cite{Zhu2015Braided} that there is a category isomorphism
$_{H}^{H}\mathcal{YD}\cong{}_{H}^{H_{R}}\mathcal{\mathcal{M}}$. Each $V\in
{}_{H}^{H}\mathcal{YD}$ is a left $H_{R}$-comodule via $\rho_{R}:V\rightarrow
H_{R}\otimes V$ given by
\begin{equation}
\rho_{R}\left(  v\right)  =v_{\left\langle -1\right\rangle }S\left(
R^{2}\right)  \otimes R^{1}v_{\left\langle 0\right\rangle },\text{ for all
}v\in V,\label{eq_rho_R}%
\end{equation}
and $\left(  V,\rho_{R}\right)  $ generates a subcoalgebra $D_{V}$ of $H_{R}$,
which is also an $H$-submodule of $H_{R}$.

Now assume further that $H$ is semisimple, then $H$ is cosemisimple and
$S^{2}=id_{H} $ (\cite{Larson1988Semisimple}). In this case, $H_{R}$ is
cosemisimple, and thus the category $_{H}^{H_{R}}\mathcal{M}$ is semisimple.

\begin{lemma}
[{\cite[Proposition 3.5]{LiuZhu2019On}}]\label{lem_decom_of_H}Let $\left(
H,R\right)  $ be a semisimple quasi-triangular Hopf algebra over $k$. Then
each Yetter-Drinfeld submodule of $\left(  H,\cdot_{ad},\Delta\right)  \in
{}_{H}^{H}\mathcal{YD}$ is a subcoalgebra of $H_{R}$, and $H$ admits a unique
decomposition into the direct sum of simple Yetter-Drinfeld submodules, while
this decomposition coincides with the direct sum $H=D_{1}\oplus\cdots\oplus
D_{r}$ of the minimal $H$-module subcoalgebras of $H_{R}$. In addition,
$_{H}^{H_{R}}\mathcal{\mathcal{M=}}{}_{H}^{D_{1}}\mathcal{\mathcal{M}}%
\oplus\cdots\oplus{}_{H}^{D_{r}}\mathcal{\mathcal{M}}$ is a direct sum of
indecomposable module categories over the tensor category {}$_{H}\mathcal{M}$.
\end{lemma}

Let $W$ be a left $H_{R}$-comodule. We use $\rho\left(  w\right)
=w^{\left\langle -1\right\rangle }\otimes w^{\left\langle 0\right\rangle }$ to
denote the left $H_{R}$-coaction on $W$. Then $H\otimes W$ is a natural object
in ${}_{H}^{H_{R}}\mathcal{M}$ with the $H$-action and $H_{R}$-coaction given
by
\begin{align}
h^{\prime}\left(  h\otimes w\right)   & =h^{\prime}h\otimes w,\label{eq:HotW1}%
\\
\rho\left(  h\otimes w\right)   & =h_{\left(  1\right)  }\cdot_{ad}%
w^{\left\langle -1\right\rangle }\otimes h_{\left(  2\right)  }\otimes
w^{\left\langle 0\right\rangle },\label{eq:HotW2}%
\end{align}
where $h,h^{\prime}\in H$ and $w\in W$. The object $H\otimes W$ was used
in~\cite{LiuZhu2019On} to characterize the structure of irreducible
Yetter-Drinfeld modules over $H$. Moreover, $H\otimes W$ has also a natural
right $H$-comodule structure
\[
\rho^{\prime}\left(  h\otimes w\right)  =\left(  h_{(1)}\otimes w\right)
\otimes h_{(2)},
\]
which makes $H\otimes W\in{}_{H}^{H_{R}}\mathcal{M}^{H}$.

The $R$-adjoint-stable algebra of $W$ is defined to be the cotensor product
$N_{W}=W^{\ast}\square_{H_{R}}\left(  H\otimes W\right)  $, with
multiplication
\[
x\circ y=\sum_{l=1}^{n}\sum_{j=1}^{m}v_{l}^{\ast}\otimes g_{l}h_{j}%
\otimes\left\langle w_{j}^{\ast},v_{l}\right\rangle w_{j},
\]
where $x=\sum_{j=1}^{m}w_{j}^{\ast}\otimes h_{j}\otimes w_{j}$, $y=\sum
_{l=1}^{n}v_{l}^{\ast}\otimes g_{l}\otimes v_{l}$ are elements in $N_{W}$. If
$D=D_{1}\oplus D_{2}\oplus\cdots\oplus D_{r}$ is a direct sum of $H$-module
subcoalgebras, then $W$ can be written as $W=W_{1}\oplus W_{2}\oplus
\cdots\oplus W_{r}$, where $\rho\left(  W_{i}\right)  \subseteq D_{i}\otimes
W_{i}$. It is evident that
\[
N_{W}=W^{\ast}\square_{H_{R}}\left(  H\otimes W\right)  \cong\oplus_{i=1}%
^{r}W_{i}^{\ast}\square_{H_{R}}\left(  H\otimes W_{i}\right)  =\oplus
_{i=1}^{r}N_{W_{i}},\text{ }%
\]
as algebras.

Naturally, $H\otimes W$ is a left $N_{W}$-module via
\[
\left(  \sum_{j}w_{j}^{\ast}\otimes h_{j}\otimes w_{j}\right)  \cdot\left(
h\otimes w\right)  =\sum_{j}hh_{j}\otimes w_{j}\left\langle w_{j}^{\ast
},w\right\rangle .
\]
In addition, let $D=D_{H\otimes W}$, for any right $N_{W}$-module $U$,
$U\otimes_{N_{W}}\left(  H\otimes W\right)  \in{}_{H}^{D}\mathcal{M}$ with the
$H$-module and $D$-comodule structures induced by that on $H\otimes W$. Let
$V\in{}_{H}^{H_{R}}\mathcal{M}$, $W^{\ast}\square_{D}V$ is a right $N_{W}%
$-module via
\begin{equation}
\left(  \sum_{i}w_{i}^{\prime\ast}\otimes v_{i}\right)  \cdot\left(  \sum
_{j}w_{j}^{\ast}\otimes h_{j}\otimes w_{j}\right)  =\sum_{i}\sum_{j}%
w_{j}^{\ast}\otimes h_{j}v_{i}\left\langle w_{i}^{\prime\ast},w_{j}%
\right\rangle ,\label{W*_cpd_V_as_N_W_module}%
\end{equation}
for $\sum_{i}w_{i}^{\prime\ast}\otimes v_{i}\in W^{\ast}\square_{D}V$,
$\sum_{j}w_{j}^{\ast}\otimes h_{j}\otimes w_{j}\in N_{W}$.

\begin{lemma}
[cf. \cite{LiuZhu2019On}]\label{lem_stru_thm}Let $W$ be a nonzero left $H_{R}%
$-comodule and $D=D_{H\otimes W}$. Then

\begin{enumerate}
\item the functors $W^{\ast}\square_{D} \bullet:{}_{H}^{D}\mathcal{M}%
\rightarrow\mathcal{M}_{N_{W}}$ and $\bullet\otimes_{N_{W}}\left(  H\otimes
W\right)  :\mathcal{M}_{N_{W}}\rightarrow{}_{H}^{D}\mathcal{M}$ induce a
category equivalence. Especially, for any $V\in{}_{H}^{D}\mathcal{M}$,
\[
V\cong\left(  W^{\ast}\square_{D}V\right)  \otimes_{N_{W}}\left(  H\otimes
W\right)  \text{.}%
\]

\item If $D$ is a minimal $H$-module subcoalgebra of $H_{R}$, then $\dim
N_{W}\dim D=\dim H\left(  \dim W\right) ^{2}$.
\end{enumerate}
\end{lemma}

Let $\left(  H,R\right)  $ be a semisimple quasi-triangular Hopf algebra over
$k$, $W$ be a finite dimensional nonzero left $H_{R}$-comodule, and
$D=D_{H\otimes W}$. If $D$ is a minimal $H$-module subcoalgebra of $H_{R}$,
then the algebras $N_{W}$ and $N_{D}$ are closely connected.

\begin{proposition}
\label{prop-matrix-equivalent} Let $D$ be a minimal $H$-module subcoalgebra of
$H_{R}$, and $W$ be a finite dimensional nonzero left $D$-comodule. There
exists a pair of positive integers $r$ and $s$, such that as algebras
\[
N_{D}\otimes M_{s}\left(  k\right)  \cong N_{W}\otimes M_{r}\left(  k\right)
.
\]

\end{proposition}

\begin{proof}
Note that $H\otimes D$ is an object of $_{H}^{D}\mathcal{M}$ with its
structure given by (\ref{eq:HotW1}) and (\ref{eq:HotW2}). By
Lemma~\ref{lem_stru_thm},
\[
H\otimes D\cong\left(  W^{\ast}\square_{D}\left(  H\otimes D\right)  \right)
\otimes_{N_{W}}\left(  H\otimes W\right)  .
\]
We view $N_{W}=W^{\ast}\square_{D}\left(  H\otimes W\right)  $ and $W^{\ast
}\square_{D}\left(  H\otimes D\right)  $ as right $H$-comodules with coactions
coming from the comultiplication on $H$. By \cite{LiuZhu2019On}, $N_{W}$ is an
$H$-simple comodule algebra. Since $W^{\ast}\square_{D}\left(  H\otimes
D\right)  \in\mathcal{M}_{N_{W}}^{H^{op}}$, by Lemma~\ref{lemma_skryabin}
there exists a pair of positive integers $r$ and $s$ such that
\[
\left(  W^{\ast}\square_{D}\left(  H\otimes D\right)  \right) ^{s}\cong\left(
N_{W}\right) ^{r}\text{ as }N_{W}\text{-modules.}%
\]
Hence,
\begin{equation}
\left(  H\otimes D\right) ^{s}\cong\left(  W^{\ast}\square_{D}\left(  H\otimes
D\right)  \right) ^{s}\otimes_{N_{W}}\left(  H\otimes W\right)  \cong\left(
H\otimes W\right) ^{r}.\label{HotD=HotW}%
\end{equation}
By~\cite{LiuZhu2019On},
\[
N_{W}\cong\operatorname{End}_{H}^{D}\left(  H\otimes W\right)  \text{ and
}N_{D}\cong\operatorname{End}_{H}^{D}\left(  H\otimes D\right)  ,
\]
thus (\ref{HotD=HotW}) implies that
\[
N_{D}\otimes M_{s}\left(  k\right)  \cong N_{W}\otimes M_{r}\left(  k\right)
.
\]

\end{proof}

When we take $W$ to be the whole $H$-module coalgebra $D$, we get a more
explicit structure on $N_{D}$.

Let $D$ be an $H$-module subcoalgebra of $H_{R}$. Then the convolution algebra
$D^{\ast}$ is a right $H$-module (left $H^{op}$-module) algebra via the right
$H$-module action given by $\left\langle d^{\ast}\leftharpoonup
\!\!\!\leftharpoonup h,d\right\rangle =\left\langle d^{\ast},h\cdot
_{ad}d\right\rangle $, for $h\in H$, $d\in D$, $d^{\ast}\in D^{\ast}$.

\begin{proposition}
\label{Prop_ND_is_WHA}Let $D$ be an $H$-module subcoalgebra of $H_{R}$. Then
the $R$-adjoint-stable algebra $N_{D}$ of $D$ is isomorphic to the smash
product $D^{\ast}\#H^{op}$.
\end{proposition}

\begin{proof}
Let
\[
\Psi:N_{D}\rightarrow D^{\ast}\#H^{op},\ \Psi\left(  \sum_{i}d_{i}^{\ast
}\otimes h_{i}\otimes d_{i}\right)  =\sum_{i}\varepsilon\left(  d_{i}\right)
d_{i}^{\ast}\leftharpoonup\!\!\!\leftharpoonup h_{i\left(  1\right)
}\#h_{i\left(  2\right)  }%
\]
and
\[
\Phi:D^{\ast}\#H^{op}\longrightarrow N_{D},\ \Phi\left(  d^{\ast}\#h\right)
=d^{\ast\left\langle 0\right\rangle }\leftharpoonup\!\!\!\leftharpoonup
S\left(  h_{\left(  1\right)  }\right)  \otimes h_{\left(  2\right)  }\otimes
d^{\ast\left\langle 1\right\rangle },
\]
where $\rho\left(  d^{\ast}\right)  =d^{\ast\left\langle 0\right\rangle
}\otimes d^{\ast\left\langle 1\right\rangle }$ is the right $D$-comodule
coaction on $D^{\ast}$ arising from the left coproduct of $D$.

Since for all $h\in H,d^{\ast}\in D^{\ast}$,%
\[
\left(  d^{\ast}\leftharpoonup\!\!\!\leftharpoonup h\right) ^{\left\langle
0\right\rangle }\otimes\left(  d^{\ast}\leftharpoonup\!\!\!\leftharpoonup
h\right) ^{\left\langle 1\right\rangle }=\left(  d^{\ast\left\langle
0\right\rangle }\leftharpoonup\!\!\!\leftharpoonup h_{\left(  2\right)
}\right)  \otimes S\left(  h_{\left(  1\right)  }\right)  \cdot_{ad}%
d^{\ast\left\langle 1\right\rangle },\text{ }%
\]
$\Phi$ is well-defined. It is routine to verify that the $\Phi$ is the inverse
of $\Psi$.

For any $h,h^{\prime}\in H$, $d^{\ast},d^{\prime\ast}\in D^{\ast}$,
\begin{align*}
& \Phi\left(  \left(  d^{\ast}\#h\right)  \left(  d^{\prime\ast}\#h^{\prime
}\right)  \right) \\
& =\Phi\left(  d^{\ast}\ast_{R}\left(  d^{\prime\ast}\leftharpoonup
\!\!\!\leftharpoonup h_{\left(  1\right)  }\right)  \#h^{\prime}h_{\left(
2\right)  }\right) \\
& =\left(  d^{\ast}\ast_{R}\left(  d^{\prime\ast}\leftharpoonup
\!\!\!\leftharpoonup h_{\left(  1\right)  }\right)  \right) ^{\left\langle
0\right\rangle }\leftharpoonup\!\!\!\leftharpoonup S\left(  h_{\left(
1\right)  }^{\prime}h_{\left(  2\right)  }\right)  \otimes h_{\left(
2\right)  }^{\prime}h_{\left(  3\right)  }\otimes\left(  d^{\ast}\ast
_{R}\left(  d^{\prime\ast}\leftharpoonup\!\!\!\leftharpoonup h_{\left(
1\right)  }\right)  \right) ^{\left\langle 1\right\rangle }
\end{align*}
\begin{align*}
& =\left(  d^{\ast\left\langle 0\right\rangle }\ast_{R}\left(  d^{\prime\ast
}\leftharpoonup\!\!\!\leftharpoonup h_{\left(  1\right)  }\right)  \right)
\leftharpoonup\!\!\!\leftharpoonup S\left(  h_{\left(  1\right)  }^{\prime
}h_{\left(  2\right)  }\right)  \otimes h_{\left(  2\right)  }^{\prime
}h_{\left(  3\right)  }\otimes d^{\ast\left\langle 1\right\rangle }\\
& =\left\langle d^{\ast\left\langle 0\right\rangle }\leftharpoonup
\!\!\!\leftharpoonup S\left(  h_{\left(  1\right)  }\right)  ,d^{\prime
\ast\left\langle 1\right\rangle }\right\rangle \left(  d^{^{\prime}%
\ast\left\langle 0\right\rangle }\leftharpoonup\!\!\!\leftharpoonup S\left(
h_{\left(  1\right)  }^{\prime}\right)  \otimes h_{\left(  2\right)  }%
^{\prime}h_{\left(  2\right)  }\otimes d^{\ast\left\langle 1\right\rangle
}\right) \\
& =\left(  d^{\ast\left\langle 0\right\rangle }\leftharpoonup
\!\!\!\leftharpoonup S\left(  h_{\left(  1\right)  }\right)  \otimes
h_{\left(  2\right)  }\otimes d^{\ast\left\langle 1\right\rangle }\right)
\left(  d^{^{\prime}\ast\left\langle 0\right\rangle }\leftharpoonup
\!\!\!\leftharpoonup S\left(  h_{\left(  1\right)  }^{\prime}\right)  \otimes
h_{\left(  2\right)  }^{\prime}\otimes d^{\prime\ast\left\langle
1\right\rangle }\right) \\
& =\Phi\left(  d^{\ast}\#h\right)  \Phi\left(  d^{\prime\ast}\#h^{\prime
}\right)  .
\end{align*}
Obviously,
\[
\Phi\left(  \varepsilon_{D}\otimes1_{H}\right)  =\sum\varepsilon_{D}%
^{<0>}\otimes1_{H}\otimes\varepsilon_{D}^{<1>}=1_{N_{D}},
\]
so $\Phi$ is an algebra isomorphism.
\end{proof}

Let $H$ be a Hopf algebra with invertible antipode $S$, and $A$ be an
$H$-module algebra. Then $A^{\ast}$ is a right $H$-module via the transpose
action $\vartriangleleft$, and $A$ is a natural right $A\#H$-module via
$a\cdot(b\#h)=S^{-1}\left(  h\right)  \cdot\left(  ab\right)  $, where
$a,\ b\in A$ and $h\in H$. This action induces a left $A\#H$-module structure
on $A^{\ast}$, whence we have a natural algebra map $\theta:A\#H\rightarrow
\operatorname{End}A^{\ast}$. Explicitly speaking, $\theta\left(  a\#h\right)
\left(  b^{\ast}\right)  =a\rightharpoonup\left(  b^{\ast}\vartriangleleft
S^{-1}\left(  h\right)  \right)  $, where $a\in A$, $h\in H$, $b^{\ast}\in
A^{\ast}$.

\begin{lemma}
\label{le:ThetaMap}The map
\begin{equation}
{\Theta:A\#H\rightarrow\operatorname{End}A^{\ast}\otimes H,\ }\text{{sending}%
}{\ a\#h}\text{\ to }{\theta(a\#h_{\left(  1\right)  })\otimes h_{\left(
2\right)  },}\label{ThetaMap}%
\end{equation}
is an algebra embedding.
\end{lemma}

\section{A Weak Hopf Algebra Structure on $A\#H$}

\label{A.H_IS_WHA}

If $\left(  H,R\right)  $ is a semisimple quasi-triangular Hopf algebra and
$A$ is a semisimple quantum commutative $H$-module algebra, then Cohen and
Westreich~\cite{Cohen1994Supersymmetry} showed that the left $A\#H$-module
category $_{A\#H}\operatorname{Mod}$ is a monoidal category.

In the special case when $\left(  H,R\right)  =\left(  kG,1\otimes1\right)  $,
we can see more structural construction on $A\#H$.

\begin{example}
\label{QC_algebra_over_kG}Let $k$ be an algebraically closed field, $\left(
H,R\right)  =\left(  kG,1\otimes1\right)  $ be the group algebra of a finite
group $G$ and $A$ be a separable quantum commutative $H$-module algebra. Then
$A$ is commutative and $A=ke_{1}\oplus ke_{2}\oplus\cdots\oplus ke_{t}$, where
$X=\left\{  e_{1},e_{2},\ldots,e_{t}\right\}  $ is the complete set of minimal
orthogonal idempotents of $A$. Clearly, for any $g\in G$, $\left\{  g\cdot
e_{1},g\cdot e_{2},\ldots,g\cdot e_{t}\right\}  $ is a permutation of $X$ and
$X$ is a disjoint union of $G$-orbits.

If $A$ is $H$-simple, then $G$ acts on $X$ transitively. Let $G_{1}=\left\{
g\in G\mid g\cdot e_{1}=e_{1}\right\}  $ be the stabilizer of $e_{1}$. Let
$g_{1},\ldots,g_{t}$ be any system of left coset representatives of $G_{1}$ in
$G$, then $\left\{  g_{1}\cdot e_{1},g_{2}\cdot e_{1},\ldots,g_{t}\cdot
e_{1}\right\}  =X$. Write $E_{ij}=g_{i}\cdot e_{1}\#g_{i}g_{j}^{-1}\in A\#H$,
$1\leq i,j\leq t$. Then one can check that the set $Y=\left\{  E_{ij}\mid1\leq
i,j\leq t\right\}  $ is a set of matrix units in $A\#H$. The centralizer
$C_{A\#H}\left(  Y\right)  $ of $Y$ in $A\#H$ is
\[
C_{A\#H}\left(  Y\right)  =\operatorname{Span}\left\{  \sum_{i=1}^{t}%
g_{i}\cdot e_{1}\#g_{i}hg_{i}^{-1}\mid h\in G_{1}\right\}
\]
and for any $h,h^{\prime}\in G_{1}$,%
\[
\left(  \sum_{i=1}^{t}g_{i}\cdot e_{1}\#g_{i}hg_{i}^{-1}\right)  \left(
\sum_{j=1}^{t}g_{j}\cdot e_{1}\#g_{j}h^{\prime}g_{j}^{-1}\right)  =\sum
_{i=1}^{t}\left(  g_{i}\cdot e_{1}\right)  \#g_{i}hh^{\prime}g_{i}%
^{-1}\text{.}%
\]
Thus, $C_{A\#H}\left(  Y\right)  \cong kG_{1}$ and as an algebra $A\#H\cong
M_{t}\left(  k\right)  \otimes kG_{1}\cong\operatorname{End}A^{\ast}%
\otimes\left(  kG_{1}\right)  $, where we note that $kG_{1}$ is a Hopf
subalgebra of $H=kG$.
\end{example}

We can see in this example that under certain assumption, the smash product
$A\#H$ is related to a Hopf subalgebra of $H$. For general quasi-triangular
Hopf algebras, there is also a similar formulating strategy.

Assume that $\left(  H,R\right)  $ is a quasi-triangular Hopf algebra, and $A$
is a quantum commutative $H$-module algebra. Recall
from~\cite{Cohen1994Supersymmetry} that any left $A\#H$-module $M$ is an
$A$-$A$-bimodule with $A$-actions $a\cdot m=\left(  a\#1_{H}\right)  m$ and
$m\cdot a=\left(  \left(  R^{2}\cdot a\right)  \#R^{1}\right)  m$. If $M$ and
$N$ are left $A\#H$-modules, then $M\otimes_{A}N$ is also a left $A\#H$-module
via
\begin{equation}
\left(  a\#h\right)  \left(  m\otimes n\right)  =\left(  a\#h_{\left(
1\right)  }\right)  m\otimes\left(  1_{A}\#h_{\left(  2\right)  }\right)
n,\label{MonoidalStruOnAsmashH}%
\end{equation}
where $a\in A,\ h\in H,\ m\in M,\ n\in N$. This makes the category
$_{A\#H}\operatorname{Mod}$ of left $A\#H$-modules into a monoidal category.

In this section, we work on a semisimple quasi-triangular Hopf algebra
$\left(  H,R\right)  $ and a strongly separable quantum commutative $H$-module
algebra $A$. We present weak Hopf algebra structures on $A\#H$ and
${\operatorname{End}}\,A^{\ast}\otimes H$ such that the map $\Theta
:A\#H\rightarrow{\operatorname{End}}\,A^{\ast}\otimes H$ (in
Lemma~\ref{le:ThetaMap}) is a weak Hopf algebra embedding if the action of the
Drinfeld element $u$ on $A$ is trivial. The natural monoidal category
${}_{A\#H}\mathcal{M}$ is precisely the full monoidal subcategory of
$_{A\#H}\operatorname{Mod}$.

We start with a semisimple quasi-triangular Hopf algebra $\left(  H,R\right)
$ and a strongly separable quantum commutative left $H$-module algebra $A$,
where strongly separable means that $A$ has a symmetric separability
idempotent $x=x^{1}\otimes x^{2}$. Let $\alpha_{A}\in A^{\ast}$ be the element
determined by $\left\langle \alpha_{A},x^{1}\right\rangle x^{2}=1_{A} $. It is
routine to verify the following statements.

\begin{lemma}
\label{lem_properties_for_x&alpha} The map $\alpha_{A}$ is the trace function
of the left regular representation of $A$. For all $a\in A,a^{\ast}\in
A^{\ast},h\in H$,

\begin{enumerate}
\item $x^{1}\left\langle x^{2}\rightharpoonup\alpha_{A},a\right\rangle =a$,
and $\left\langle a^{\ast},x^{1}\right\rangle x^{2}\rightharpoonup\alpha
_{A}=a^{\ast}$, i.e. $x^{1}\otimes\left( x^{2}\rightharpoonup\alpha_{A}\right)
$ is a pair of dual basis for $A$ and $A^{*}$;

\item $h\cdot x^{1}\otimes x^{2}=x^{1}\otimes S\left(  h\right)  \cdot x^{2}$,
provided that $H$ is semisimple and involutory;

\item $\left\langle \alpha_{A},h\cdot a\right\rangle =\varepsilon_{H}\left(
h\right)  \left\langle \alpha_{A},a\right\rangle $.
\end{enumerate}
\end{lemma}

\begin{proposition}
\label{prop_wba_stru} The smash product $A\#H$ is a weak bialgebra, with the
comultiplication and the counit map given by
\begin{equation}
\tilde{\Delta}\left(  a\#h\right)  =a\left(  R^{2}\cdot x^{1}\right)
\#R^{1}h_{\left(  1\right)  }\otimes x^{2}\#h_{\left(  2\right)  }%
,\ \tilde{\varepsilon}\left(  a\#h\right)  =\left\langle \alpha_{A}%
,a\right\rangle \varepsilon\left(  h\right)  ,\label{eq:deltaofAsmashH}%
\end{equation}
where $a\in A,\ h\in H$.
\end{proposition}

Write $\tilde{1}_{\left(  1\right)  }\otimes\tilde{1}_{\left(  2\right)
}=\tilde{\Delta}\left(  1_{A}\otimes1_{H}\right)  =R^{2}\cdot x^{1}%
\#R^{1}\otimes x^{2}\#1_{H}$. We use the notation $x_{i}={x_{i}}^{1}%
\otimes{x_{i}}^{2}=x,\allowbreak\ \forall~i\in\mathbb{N}^{+}$.

\begin{lemma}
For any $a,\ b\in A$, $h\in H$,
\begin{align}
& \left(  a\#1_{H}\right)  \left(  R^{2}\cdot b\#R^{1}\right)  =\left(
R^{2}\cdot b\#R^{1}\right)  \left(  a\#1_{H}\right)  ;\label{eq1.1}\\
& \left(  {R_{1}}^{2}\cdot a\#{R_{1}}^{1}\right)  \left(  {R_{2}}^{2}\cdot
b\#{R_{2}}^{1}\right)  =R^{2}\cdot\left(  ba\right)  \#R^{1};\label{eq1.2}\\
& \tilde{\Delta}\left(  a\#h\right)  =\tilde{1}_{\left(  1\right)  }\left(
a\#h_{\left(  1\right)  }\right)  \otimes\tilde{1}_{\left(  2\right)  }\left(
1_{A}\#h_{\left(  2\right)  }\right)  ;\label{eq1.3}\\
& \tilde{1}_{\left(  1\right)  }\left(  1_{A}\#h_{\left(  1\right)  }\right)
\otimes\tilde{1}_{\left(  2\right)  }\left(  1_{A}\#h_{\left(  2\right)
}\right)  =\left(  1_{A}\#h_{\left(  1\right)  }\right)  \tilde{1}_{\left(
1\right)  }\otimes\left(  1_{A}\#h_{\left(  2\right)  }\right)  \tilde
{1}_{\left(  2\right)  }.\label{eq1.4}%
\end{align}

\end{lemma}

\begin{proof}
The equalities (\ref{eq1.1}) and (\ref{eq1.2}) follow from the fact that $A$
is quantum commutative, while (\ref{eq1.3}) follows from (\ref{eq1.1}). Note
that by 3) of Lemma~\ref{lem_properties_for_x&alpha}, for each $h\in H$
\[
h_{\left(  1\right)  }\cdot x^{1}\otimes h_{\left(  2\right)  }\cdot
x^{2}=\varepsilon\left(  h\right)  x^{1}\otimes x^{2},
\]
so
\begin{align*}
\left(  1_{A}\#h_{\left(  1\right)  }\right)  \tilde{1}_{\left(  1\right)
}\otimes\left(  1_{A}\#h_{\left(  2\right)  }\right)  \tilde{1}_{\left(
2\right)  }  & =\left(  h_{\left(  1\right)  }R^{2}\right)  \cdot
x^{1}\#h_{\left(  2\right)  }R^{1}\otimes h_{\left(  3\right)  }\cdot
x^{2}\#h_{\left(  4\right)  }\\
& =\left(  R^{2}h_{\left(  2\right)  }\right)  \cdot x^{1}\#R^{1}h_{\left(
1\right)  }\otimes h_{\left(  3\right)  }\cdot x^{2}\#h_{\left(  4\right)  }\\
& =\tilde{1}_{\left(  1\right)  }\left(  1_{A}\#h_{\left(  1\right)  }\right)
\otimes\tilde{1}_{\left(  2\right)  }\left(  1_{A}\#h_{\left(  2\right)
}\right)  .
\end{align*}
\bigskip
\end{proof}

\begin{proof}
[ Proof of Proposition~\ref{prop_wba_stru}] By (\ref{eq1.2}), $\left(
\tilde{1}_{\left(  1\right)  }\otimes\tilde{1}_{\left(  2\right)  }\right)
^{2}=\tilde{1}_{\left(  1\right)  }\otimes\tilde{1}_{\left(  2\right)  }$,
thus for any $a,b\in A$ and $h,g\in H$,%
\begin{align*}
\tilde{\Delta}\left(  a\#h\right)  \tilde{\Delta}\left(  b\#g\right)   &
=\tilde{1}_{\left(  1\right)  }\left(  a\#h_{\left(  1\right)  }\right)
\tilde{1}_{\left(  1^{\prime}\right)  }\left(  b\#g_{\left(  1\right)
}\right)  \otimes\tilde{1}_{\left(  2\right)  }\left(  1_{A}\#h_{\left(
2\right)  }\right)  \tilde{1}_{\left(  2^{\prime}\right)  }\left(
1_{A}\#g_{\left(  2\right)  }\right) \\
& =\tilde{1}_{\left(  1\right)  }\tilde{1}_{\left(  1^{\prime}\right)
}\left(  a\#h_{\left(  1\right)  }\right)  \left(  b\#g_{\left(  1\right)
}\right)  \otimes\tilde{1}_{\left(  2\right)  }\tilde{1}_{\left(  2^{\prime
}\right)  }\left(  1_{A}\#h_{\left(  2\right)  }\right)  \left(
1_{A}\#g_{\left(  2\right)  }\right) \\
& =\tilde{1}_{\left(  1\right)  }\left(  a\left(  h_{\left(  1\right)  }\cdot
b\right)  \#h_{\left(  2\right)  }g_{\left(  1\right)  }\right)  \otimes
\tilde{1}_{\left(  2\right)  }\left(  1_{A}\#h_{\left(  3\right)  }g_{\left(
2\right)  }\right) \\
& =\tilde{\Delta}\left(  \left(  a\#h\right)  \left(  b\#g\right)  \right)  .
\end{align*}
By (\ref{eq1.1}),
\begin{align*}
\tilde{\Delta}\left(  \tilde{1}_{\left(  1\right)  }\right)  \otimes\tilde
{1}_{\left(  2\right)  }  & =\left(  {R_{1}}^{2}{R_{3}}^{2}\cdot{x_{1}}%
^{1}\right)  \left(  {R_{2}}^{2}\cdot{x_{2}}^{1}\right)  \#{R_{2}}^{1}{R_{1}%
}^{1}\otimes x{_{2}}^{2}\#{R_{3}}^{1}\otimes{x_{1}}^{2}\#1_{H}\\
& =\left(  {R_{1}}^{2}\cdot\left(  \left(  {R_{3}}^{2}\cdot{x_{1}}^{1}\right)
{x_{2}}^{1}\right)  \right)  \#{R_{1}}^{1}\otimes x{_{2}}^{2}\#{R_{3}}%
^{1}\otimes{x_{1}}^{2}\#1_{H}\\
& =\left(  {R_{1}}^{2}\cdot{x_{2}}^{1}\right)  \#{R_{1}}^{1}\otimes x{_{2}%
}^{2}\left(  {R_{3}}^{2}\cdot{x_{1}}^{1}\right)  \#{R_{3}}^{1}\otimes{x_{1}%
}^{2}\#1_{H}\\
& =\tilde{1}_{\left(  1\right)  }\otimes\tilde{1}_{\left(  2\right)  }%
\tilde{1}_{\left(  1^{\prime}\right)  }\otimes\tilde{1}_{\left(  2^{\prime
}\right)  }\\
& =\tilde{1}_{\left(  1\right)  }\otimes\tilde{\Delta}\left(  \tilde
{1}_{\left(  2\right)  }\right)  .
\end{align*}
Since for any $a\in A$ and $h\in H$,
\begin{align*}
\left(  \tilde{\Delta}\otimes id_{A}\right)  \tilde{\Delta}\left(
a\#h\right)   & =\tilde{\Delta}\left(  \tilde{1}_{\left(  1\right)  }\right)
\tilde{\Delta}\left(  a\#h_{\left(  1\right)  }\right)  \otimes\tilde
{1}_{\left(  2\right)  }\left(  1\#h_{\left(  2\right)  }\right) \\
& =\tilde{1}_{\left(  1\right)  }\left(  a\#h_{\left(  1\right)  }\right)
\otimes\tilde{1}_{\left(  2\right)  }\tilde{1}_{\left(  1^{\prime}\right)
}\left(  1\#h_{\left(  2\right)  }\right)  \otimes\tilde{1}_{\left(
2^{\prime}\right)  }\left(  1\#h_{\left(  3\right)  }\right) \\
& =\left(  id_{A}\otimes\tilde{\Delta}\right)  \tilde{\Delta}\left(
a\#h\right)  ,
\end{align*}
$\tilde{\Delta}$ is coassociative.

It is routine to check the counit axiom. To see that $A\#H$ is a weak
bialgebra, we need to verify the remaining axioms. For any $a,\ b,\ c\in A$,
and $h,\ g,\ f\in H,$
\begin{align*}
\tilde{\varepsilon}\left(  \left(  a\#h\right)  \left(  b\#g\right)  \left(
c\#f\right)  \right)   & =\tilde{\varepsilon}\left(  a\left(  h_{\left(
1\right)  }\cdot b\right)  \left(  h_{\left(  2\right)  }g_{\left(  1\right)
}\cdot c\right)  \#h_{\left(  3\right)  }g_{\left(  2\right)  }f\right) \\
& =\left\langle \alpha_{A},a\left(  h\cdot\left(  b\left(  g\cdot c\right)
\right)  \right)  \right\rangle \varepsilon\left(  f\right)  ,
\end{align*}
so
\begin{align*}
\tilde{\varepsilon}\left(  \left(  a\#h\right)  \left(  b\#g\right) _{\left(
1\right)  }\right)  \tilde{\varepsilon}\left(  \left(  b\#g\right) _{\left(
2\right)  }\left(  c\#f\right)  \right)   & =\tilde{\varepsilon}\left(
\left(  a\#h\right)  \left(  b\left(  R^{2}\cdot x^{1}\right)  \#R^{1}%
g_{\left(  1\right)  }\right)  \right)  \tilde{\varepsilon}\left(  \left(
x^{2}\#g_{\left(  2\right)  }\right)  \left(  c\#f\right)  \right) \\
& =\left\langle \alpha_{A},a\left(  h\cdot\left(  bx^{1}\right)  \right)
\right\rangle \left\langle \alpha_{A},x^{2}\left(  g\cdot c\right)
\right\rangle \varepsilon\left(  f\right) \\
& =\left\langle \alpha_{A},a\left(  h\cdot\left(  b\left(  g\cdot c\right)
\right)  \right)  \right\rangle \varepsilon\left(  f\right) \\
& =\tilde{\varepsilon}\left(  \left(  a\#h\right)  \left(  b\#g\right)
\left(  c\#f\right)  \right)  .
\end{align*}
Similarly, one can show that
\[
\tilde{\varepsilon}\left(  \left(  a\#h\right)  \left(  b\#g\right) _{\left(
1\right)  }\right)  \tilde{\varepsilon}\left(  \left(  b\#g\right) _{\left(
2\right)  }\left(  c\#f\right)  \right)  =\tilde{\varepsilon}\left(  \left(
a\#h\right)  \left(  b\#g\right)  \left(  c\#f\right)  \right)  ,
\]
therefore $A\#H$ is a weak bialgebra.
\end{proof}

The source counital map $\tilde{\varepsilon}_{s}$ and target counital map
$\tilde{\varepsilon}_{t}$ are given by
\[
\tilde{\varepsilon}_{s}\left(  a\#h\right)  =\tilde{1}_{\left(  1\right)
}\tilde{\varepsilon}\left(  \left(  a\#h\right)  \tilde{1}_{\left(  2\right)
}\right)  =R^{2}\cdot x^{1}\#R^{1}\left\langle \alpha_{A},a\left(  h\cdot
x^{2}\right)  \right\rangle =\left(  R^{2}S\left(  h\right)  \right)  \cdot
a\#R^{1},
\]
and
\[
\tilde{\varepsilon}_{t}\left(  a\#h\right)  =\varepsilon\left(  h\right)
a\#1_{H}.
\]
The target and source subalgebras of $A\#H$ are respectively
\[
\left(  A\#H\right) _{t}=A\#1_{H}\cong A
\]
and
\[
\left(  A\#H\right) _{s}=\left\{  R^{2}\cdot a\#R^{1}\mid a\in A\right\}
\cong A^{op}.
\]

\begin{theorem}
\label{Thm_WHA_stru_of_AsmH} If the action of the Drinfeld element $u=S\left(
R^{2}\right)  R^{1}$ on $A$ is trivial, then $A\#H$ is a weak Hopf algebra,
with the antipode $\tilde{S}$ given by
\begin{equation}
\tilde{S}\left(  a\#h\right)  =\left(  1_{A}\#S\left(  h\right)  \right)
\left(  R^{2}\cdot a\#R^{1}\right)  .\label{eq:antipodeofAsmashH}%
\end{equation}

Furthermore, if $A$ belongs to the M\"{u}ger center of $_{H}\mathcal{M}$,
i.e., for all $a\in A$,%
\[
\left(  R_{2}{}^{2}R_{1}{}^{1}\right)  \cdot a\otimes R_{2}{}^{1}R_{1}{}%
^{2}=a\otimes1_{H},
\]
then $\left(  A\#H,\mathcal{R}\right)  $ is a quasi-triangular with R-matrix
\[
\mathcal{R}=\tilde{1}_{\left(  2\right)  }\left(  1_{A}\#R^{1}\right)
\tilde{1}_{\left(  1^{\prime}\right)  }\otimes\tilde{1}_{\left(  1\right)
}\left(  1_{A}\#R^{2}\right)  \tilde{1}_{\left(  2^{\prime}\right)  }.
\]
If in addition that $\left(  H,R\right)  $ is triangular, then $\left(
A\#H,\mathcal{R}\right)  $ is also triangular.
\end{theorem}

\begin{proof}
First, we see that $\tilde{S}$ is an anti-algebra homomorphism. For any $a\in
A,\ h\in H,$
\begin{align*}
\tilde{S}\left(  \left(  1_{A}\#h\right)  \left(  a\#1_{H}\right)  \right)   &
=\left(  1\#S\left(  h_{\left(  2\right)  }\right)  \right)  \left(  \left(
R^{2}h_{\left(  1\right)  }\right)  \cdot a\#R^{1}\right) \\
& =R^{2}\cdot a\#R^{1}S\left(  h\right)  =\tilde{S}\left(  a\#1_{H}\right)
\tilde{S}\left(  1_{A}\#h\right)  .
\end{align*}
It is easy to check
\begin{equation}
\tilde{S}\left(  \tilde{1}_{\left(  1\right)  }\right)  \tilde{1}_{\left(
2\right)  }=\left(  S\left(  u\right)  \cdot x^{1}\#1_{H}\right)  \left(
x^{2}\#1_{H}\right)  =1_{A}\#1_{H},\label{S(1_1)1_2}%
\end{equation}%
\begin{equation}
\tilde{1}_{\left(  1\right)  }\tilde{S}\left(  \tilde{1}_{\left(  2\right)
}\right)  =\left(  {R_{1}}^{2}\cdot x^{1}\#{R_{1}}^{1}\right)  \left(  {R_{2}%
}^{2}\cdot x^{1}\#{R_{2}}^{1}\right)  =1_{A}\#1_{H}.\label{1_1S(1_2)}%
\end{equation}
For any $a\in A,\ h\in H$,
\begin{align*}
\tilde{S}\left(  \left(  a\#h\right) _{\left(  1\right)  }\right)  \left(
a\#h\right) _{\left(  2\right)  }  & =\tilde{S}\left(  \tilde{1}_{\left(
1\right)  }\left(  a\#h_{\left(  1\right)  }\right)  \right)  \tilde
{1}_{\left(  2\right)  }\left(  1_{A}\#h_{\left(  2\right)  }\right) \\
& =\left(  {R}^{2}S\left(  h\right)  \right)  \cdot a\#1_{A}=\tilde
{\varepsilon}_{s}\left(  a\#h\right)  ,\ \text{by (\ref{S(1_1)1_2})}.
\end{align*}
Similarly, by using (\ref{1_1S(1_2)}) we obtain $\left(  a\#h\right) _{\left(
1\right)  }\tilde{S}\left(  \left(  a\#h\right) _{\left(  2\right)  }\right)
=\tilde{\varepsilon}_{t}\left(  a\#h\right)  $. Hence,
\begin{align*}
& \tilde{S}\left(  \left(  a\#h\right) _{\left(  1\right)  }\right)  \left(
a\#h\right) _{\left(  2\right)  }\tilde{S}\left(  \left(  a\#h\right)
_{\left(  3\right)  }\right) \\
& =\tilde{S}\left(  \tilde{1}_{\left(  1\right)  }\left(  a\#h_{\left(
1\right)  }\right)  \right)  \tilde{1}_{\left(  2\right)  }\tilde{1}_{\left(
1^{\prime}\right)  }\left(  1\#h_{\left(  2\right)  }\right)  \tilde{S}\left(
\tilde{1}_{\left(  2^{\prime}\right)  }\left(  1\#h_{\left(  3\right)
}\right)  \right) \\
& =\tilde{S}\left(  a\#h_{\left(  1\right)  }\right)  \tilde{S}\left(
\tilde{1}_{\left(  1\right)  }\right)  \tilde{1}_{\left(  2\right)  }\tilde
{1}_{\left(  1^{\prime}\right)  }\left(  1\#h_{\left(  2\right)  }\right)
\left(  1\#S\left(  h_{\left(  3\right)  }\right)  \right)  \tilde{S}\left(
\tilde{1}_{\left(  2^{\prime}\right)  }\right) \\
& =\tilde{S}\left(  a\#h\right)  ,
\end{align*}
that is, $\tilde{S}$ is an antipode for $A\#H$.

Assume further that $\left(  R_{2}{}^{2}R_{1}{}^{1}\right)  \cdot a\otimes
R_{2}{}^{1}R_{1}{}^{2}=a\otimes1_{H}$ holds for all $a\in A$, then
\begin{equation}
{R}^{1}\cdot a\otimes R{}^{2}={R}^{2}\cdot a\otimes S\left(  R{}^{1}\right)
.\label{EqEquivToAinMCenter}%
\end{equation}
So we have%
\begin{align*}
\left(  1_{A}\#{R_{1}}^{1}\right)  \tilde{1}_{\left(  1\right)  }%
\otimes\left(  1_{A}\#{R_{1}}^{2}\right)  \tilde{1}_{\left(  2\right)  }  &
=\left(  1_{A}\#{R_{1}}^{1}\right)  \left(  {R_{2}}^{2}\cdot{x}^{1}\#{R_{2}%
}^{1}\right)  \otimes\left(  1_{A}\#{R_{1}}^{2}\right)  \left(  x^{2}%
\#1_{H}\right) \\
& =\left(  \left(  {R_{1}}^{1}{R_{2}}^{2}\right)  \cdot{x}^{1}\#{R_{3}}%
^{1}{R_{2}}^{1}\right)  \otimes\left(  1_{A}\#{R_{1}}^{2}{R_{3}}^{2}\right)
\left(  x^{2}\#1_{H}\right) \\
& ={R_{1}}^{1}\cdot{x}^{1}\#{R_{3}}^{1}{R_{4}}^{1}{R_{2}}^{1}\otimes\left(
1_{A}\#{R_{1}}^{2}\right)  \left(  \left(  {R_{4}}^{2}S({R_{2}}^{2})\right)
\cdot x^{2}\#{R_{3}}^{2}\right) \\
& =\left(  {R_{1}}^{1}{R_{2}}^{1}\right)  \cdot{x}^{1}\#{R_{3}}^{1}%
\otimes\left(  {R_{2}}^{2}\cdot x^{2}\#{R_{1}}^{2}{R_{3}}^{2}\right) \\
& ={x}^{1}\#{R_{3}}^{1}\otimes\left(  {R_{1}}^{1}\cdot x^{2}\#S\left(  {R_{1}%
}^{2}\right)  {R_{3}}^{2}\right) \\
& ={x}^{2}\#{R_{3}}^{1}\otimes\left(  {R_{1}}^{2}\cdot x^{1}\#{R_{1}}%
^{1}{R_{3}}^{2}\right) \\
& =\tilde{1}_{\left(  2\right)  }\left(  1_{A}\#{R_{1}}^{1}\right)
\otimes\tilde{1}_{\left(  1\right)  }\left(  1_{A}\#{R_{1}}^{2}\right)  ,
\end{align*}
and it follows that
\begin{equation}
\mathcal{R=}\left(  1_{A}\#{R_{1}}^{1}\right)  \tilde{1}_{\left(  1\right)
}\otimes\left(  1_{A}\#{R_{1}}^{2}\right)  \tilde{1}_{\left(  2\right)
}=\tilde{1}_{\left(  2\right)  }\left(  1_{A}\#{R_{1}}^{1}\right)
\otimes\tilde{1}_{\left(  1\right)  }\left(  1_{A}\#{R_{1}}^{2}\right)
.\label{SimplifyR}%
\end{equation}
Let $\mathcal{\bar{R}=}\tilde{1}_{\left(  1\right)  }\left(  1_{A}\#S\left(
R^{1}\right)  \right)  \tilde{1}_{\left(  2^{\prime}\right)  }\otimes\tilde
{1}_{\left(  2\right)  }\left(  1_{A}\#R^{2}\right)  \tilde{1}_{\left(
1^{\prime}\right)  }$. We see that
\[
\mathcal{\bar{R}=}\tilde{1}_{\left(  1\right)  }\left(  1_{A}\#S(R^{1}%
)\right)  \otimes\tilde{1}_{\left(  2\right)  }\left(  1_{A}\#R^{2}\right)
=\left(  1_{A}\#S(R^{1})\right)  \tilde{1}_{\left(  2\right)  }\otimes\left(
1_{A}\#R^{2}\right)  \tilde{1}_{\left(  1\right)  }.
\]
Thus $\bar{\mathcal{R}}\mathcal{R}=\tilde{\Delta}\left(  \tilde{1}\right)  $
and $\mathcal{R}\bar{\mathcal{R}}=\tilde{\Delta}^{cop}\left(  \tilde
{1}\right)  $.

The identities $\left(  id\otimes\tilde{\Delta}\right)  \mathcal{R}%
=\mathcal{R}^{13}\mathcal{R}^{12}$ and $\left(  \tilde{\Delta}\otimes
id\right)  \mathcal{R}=\mathcal{R}^{13}\mathcal{R}^{23}$ are obvious. One can
check by using (\ref{EqEquivToAinMCenter}), that
\[
\tilde{1}_{\left(  1\right)  }\left(  a\#1_{H}\right)  \otimes\tilde
{1}_{\left(  2\right)  }=\tilde{1}_{\left(  1\right)  }\left(  1_{A}%
\#R^{2}\right)  \otimes\tilde{1}_{\left(  2\right)  }\left(  \left(
R^{1}\cdot a\right)  \#1_{H}\right)  .
\]
Hence for any $a\in A$ and $h\in H$,
\begin{align*}
\tilde{\Delta}^{cop}\left(  a\#h\right)  \mathcal{R}  & \mathcal{=}\tilde
{1}_{\left(  2\right)  }\left(  1_{A}\#h_{\left(  2\right)  }\right)  \left(
1_{A}\#R^{1}\right)  \tilde{1}_{\left(  1^{\prime}\right)  }\otimes\tilde
{1}_{\left(  1\right)  }\left(  a\#h_{\left(  1\right)  }\right)  \left(
1_{A}\#R^{2}\right)  \tilde{1}_{\left(  2^{\prime}\right)  }\\
& =\tilde{1}_{\left(  2\right)  }\left(  1_{A}\#R^{1}h_{\left(  1\right)
}\right)  \tilde{1}_{\left(  1^{\prime}\right)  }\otimes\tilde{1}_{\left(
1\right)  }\left(  a\#1_{H}\right)  \left(  1_{A}\#R^{2}h_{\left(  2\right)
}\right)  \tilde{1}_{\left(  2^{\prime}\right)  }\\
& =\tilde{1}_{\left(  2\right)  }\left(  \left(  {R_{2}}^{1}\cdot a\right)
\#{R_{1}}^{1}h_{\left(  1\right)  }\right)  \tilde{1}_{\left(  1^{\prime
}\right)  }\otimes\tilde{1}_{\left(  1\right)  }\left(  1_{A}\#{R_{2}}%
^{2}\right)  \left(  1_{A}\#{R_{1}}^{2}h_{\left(  2\right)  }\right)
\tilde{1}_{\left(  2^{\prime}\right)  }\\
& =\tilde{1}_{\left(  2\right)  }\left(  1_{A}\#R^{1}\right)  \left(
a\#h_{\left(  1\right)  }\right)  \tilde{1}_{\left(  1^{\prime}\right)
}\otimes\tilde{1}_{\left(  1\right)  }\left(  1_{A}\#R^{2}\right)  \left(
1_{A}\#h_{\left(  2\right)  }\right)  \tilde{1}_{\left(  2^{\prime}\right)
}\\
& =\mathcal{R}\tilde{\Delta}\left(  a\#h\right)  .
\end{align*}
This completes the proof that $\left(  A\#H,\mathcal{R}\right)  $ is a
quasi-triangular weak Hopf algebra.
\end{proof}

With this weak Hopf algebra structure the category $_{A\#H}\mathcal{M}$ is a
multitensor category, which is a monoidal subcategory of ${}_{A\#H}%
\operatorname{Mod} $ as introduced in~\cite{Cohen1994Supersymmetry}.

As we can see in Lemma~\ref{le:ThetaMap}, there is an algebra inclusion of
$A\#H$ into $\operatorname{End}A^{\ast}\otimes H$. Observing that in the case
when $H=kG$ and an $H$-simple module algebra $A$ are as in
Example~\ref{QC_algebra_over_kG}, $A\#H\cong M_{t}\left(  kG_{1}\right)  $
with matrix units $\left\{  E_{ij}=g_{i}\cdot e_{1}\#g_{i}g_{j}^{-1}\mid1\leq
i,j\leq t\right\}  $ and the injection $\varphi:G_{1}\rightarrow A\#H$,
mapping $h\in G_{1}$ to $\sum_{i=1}^{t}g_{i}\cdot e_{1}\#g_{i}hg_{i}^{-1}$. It
is routine to check that $x=\sum_{i=1}^{t}e_{i}\otimes e_{i}$ is the symmetric
separability idempotent, and
\[
\widetilde{\Delta}\left(  E_{ij}\right)  =E_{ij}\otimes E_{ij}\text{, and
}\widetilde{\Delta}\left(  \varphi\left(  h\right)  \right)  =\varphi\left(
h\right)  \otimes\varphi\left(  h\right)  \text{, for any }h\in G_{1}\text{.}%
\]
This illustrates that $A\#H$ can embed into $\operatorname{End}A^{\ast}\otimes
H$ as weak Hopf algebras if we identify $\operatorname{End}A^{\ast}$ with the
canonical weak Hopf algebra $M_{t}\left(  k\right)  $.

So it arises a natural question if there exists a weak Hopf algebra structure
on $\operatorname{End}A^{\ast}\otimes H$ which makes $\Theta:A\#H\rightarrow
\operatorname{End}A^{\ast}\otimes H$ a weak Hopf algebra embedding when $H$ is
not a group algebra.

In cooperation with the form of $R$-adjoint-stable algebra $N_{D}=D^{\ast
}\square_{D}\left(  H\otimes D\right)  $ for an $H$-module subalgebra $D$ of
$H_{R}$, we prefer to work on $A\otimes H\otimes A^{\ast}$ instead of
$\operatorname{End}A^{\ast}\otimes H$.

Let $\left(  H,R\right)  $ be a semisimple quasi-triangular Hopf algebra and
$A $ be a strongly separable quantum commutative left $H$-module algebra $A$
with the symmetric separability idempotent $x=x^{1}\otimes x^{2}$. Let
$B=A\otimes H\otimes A^{\ast}$.

The multiplication on $B$ is
\[
\left(  a\otimes h\otimes a^{\ast}\right)  \left(  b\otimes g\otimes b^{\ast
}\right)  =\left\langle b^{\ast},a\right\rangle b\otimes hg\otimes a^{\ast
},\text{ for }a,b\in A,h,g\in H,a^{\ast},b^{\ast}\in A^{\ast},
\]
and the unit of $B$ is
\[
1_{B}=x^{1}\otimes1_{H}\otimes x^{2}\rightharpoonup\alpha_{A}.
\]
Observe that $A$ is finite dimensional, so $A^{\ast}$ is a coalgebra.

\begin{proposition}
\label{prop_wha_stru_on_B}Define the coproduct, the counit, and the antipode
via
\begin{align*}
&  \Delta_{B}\left(  a\otimes h\otimes a^{\ast}\right)  =\left(  {R_{1}}%
^{2}\cdot x^{1}\right)  a\otimes{R_{2}}^{1}h_{\left(  1\right)  }{R_{1}}%
^{1}\otimes a_{\left(  2\right)  }^{\ast}\otimes x^{2}\otimes h_{\left(
2\right)  }\otimes a_{\left(  1\right)  }^{\ast}\vartriangleleft{R_{2}}^{2},\\
&  \varepsilon_{B}\left(  a\otimes h\otimes a^{\ast}\right)  =\left\langle
\alpha_{A},a\right\rangle \varepsilon\left(  h\right)  \left\langle a^{\ast
},1_{A}\right\rangle ,\\
&  S_{B}\left(  a\otimes h\otimes a^{\ast}\right)  =\left\langle a^{\ast
}\vartriangleleft{R_{2}}^{2},{x}^{1}\right\rangle {x}^{2}\otimes{R_{1}}%
^{1}S\left(  {R_{2}}^{1}h\right)  \otimes\alpha_{A}\leftharpoonup\left(
{R_{1}}^{2}\cdot a\right)  ,
\end{align*}
for all $a\in A$, $h\in H$, $a^{\ast}\in A^{\ast}$, where $\vartriangleleft$
is the right transpose action of $H$ on $A^{\ast}$, induced by the left
$H$-action on $A$. Let
\[
R_{B}=\left(  \left(  {R_{3}}^{2}\cdot{x_{2}}^{1}\right)  {x_{1}}^{1}%
\otimes{R_{2}}^{1}{R_{3}}^{1}\otimes\alpha_{A\left(  1\right)  }%
\vartriangleleft{R_{1}}^{2}\right)  \otimes\left(  {x_{2}}^{2}\otimes{R_{1}%
}^{1}{R_{2}}^{2}\otimes x{_{1}}^{2}\rightharpoonup\alpha_{A\left(  2\right)
}\right)  \,.
\]
Then $\left(  B,R_{B}\right)  $ is a quasi-triangular weak Hopf algebra.
\end{proposition}

The proof of the proposition is fairly straightforward and tedious. We omit it here.

Note that the left and right multiplication on $A$ lead to both a right and a
left $A^{\ast}$-comodule structure on $A$. Apparently, by 2) of
Lemma~\ref{lem_properties_for_x&alpha} we have
\begin{align}
a_{\left\langle 0\right\rangle }\otimes a_{\left\langle 1\right\rangle }  &
={x}^{1}{a}\otimes x^{2}\rightharpoonup\alpha_{A},\label{right_A*-comod}\\
a_{\left\langle -1\right\rangle }\otimes a_{\left\langle 0\right\rangle }  &
=x^{2}\rightharpoonup\alpha_{A}\otimes a{x}^{1},\label{left_A*-comod}%
\end{align}
for all $a\in A$. With the given weak Hopf algebra structure on $B$, the
source counital subalgebra is
\begin{align}
B_{s}  & =\left\{  \left(  {R}^{2}\cdot a\right)  {x}^{1}\otimes{R}^{1}\otimes
x^{2}\rightharpoonup\alpha_{A}\mid a\in A\right\} \nonumber\\
& =\left\{  \left(  {R}^{2}\cdot a\right) _{\left\langle 0\right\rangle
}\otimes{R}^{1}\otimes\left(  {R}^{2}\cdot a\right) _{\left\langle
-1\right\rangle }\mid a\in A\right\}  \cong A^{op},\label{eq_Bs}%
\end{align}
and the target subalgebra is
\begin{align}
B_{t}  & =\left\{  {x}^{1}{a}\otimes1_{H}\otimes x^{2}\rightharpoonup
\alpha_{A}\mid a\in A\right\} \nonumber\\
& =\left\{  a_{\left\langle 0\right\rangle }\otimes1_{H}\otimes
a_{\left\langle 1\right\rangle }\mid a\in A\right\}  \cong A.\label{eq_Bt}%
\end{align}
Since $B_{t}$ is simple in $_{B}\mathcal{M}$, $_{B}\mathcal{M}$ is a braided
tensor category.

\begin{proposition}
As braided tensor categories, $_{B}\mathcal{M}$ is equivalent to
$_{H}\mathcal{M}$.

\begin{proof}
Note that $A^{\ast}\otimes H$ is a free right $H$-module and $B\cong
\operatorname{End}_{H}\left(  A^{\ast}\otimes H\right)  $. So the functor
$F=\left(  A^{\ast}\otimes H\right)  \otimes_{H}\bullet{}\cong A^{\ast}%
\otimes\bullet{}:{}_{H}\mathcal{M}\rightarrow{}_{B}\mathcal{M}$ is a category
equivalence. To complete the proof, it suffices to show that $F$ is a braided
tensor equivalence. $F\left(  M\right)  =A^{\ast}\otimes M\in{}_{B}%
\mathcal{M}$ is an $A$-bimodule via $a\cdot\left(  a^{\ast}\otimes m\right)
=a\rightharpoonup a^{\ast}\otimes m$ and $\left(  a^{\ast}\otimes m\right)
\cdot a=a^{\ast}\leftharpoonup\left(  {R}^{2}\cdot a\right)  \otimes{R}^{1}m $.

For $M,N\in${}${}{}{}_{H}\mathcal{M}$, define
\[
\xi_{M,N}:F\left(  M\otimes N\right)  \rightarrow F\left(  M\right)
\otimes_{A}F\left(  N\right)  ,\ \xi_{M,N}\left(  a^{\ast}\otimes\left(
m\otimes n\right)  \right)  =\left(  a^{\ast}\otimes m\right)  \otimes
_{A}\left(  \alpha_{A}\otimes n\right)  \text{.}%
\]
For all $a\in A$, $h\in H$, $a^{\ast},b^{\ast}\in A^{\ast}$, $m\in M$, $n\in
N$,
\begin{align*}
& \left(  a\otimes h\otimes a^{\ast}\right)  \cdot\xi_{M,N}\left(  b^{\ast
}\otimes\left(  m\otimes n\right)  \right) \\
& =\left(  a\otimes h\otimes a^{\ast}\right)  \cdot\left(  \left(  b^{\ast
}\otimes m\right)  \otimes_{A}\left(  \alpha_{A}\otimes n\right)  \right) \\
& =\left(  \left(  {R_{1}}^{2}\cdot x^{1}\right)  a\otimes{R_{2}}%
^{1}h_{\left(  1\right)  }{R_{1}}^{1}\otimes a_{\left(  2\right)  }^{\ast
}\right)  \cdot\left(  b^{\ast}\otimes m\right)  \otimes_{A}\left(
x^{2}\otimes h_{\left(  2\right)  }\otimes a_{\left(  1\right)  }^{\ast
}\vartriangleleft{R_{2}}^{2}\right)  \cdot\left(  \alpha_{A}\otimes n\right)
\\
& =\left\langle b^{\ast},\left(  {R_{1}}^{2}\cdot x^{1}\right)  a\right\rangle
\left\langle \alpha_{A},x^{2}\right\rangle \left(  a_{\left(  2\right)
}^{\ast}\otimes{R_{2}}^{1}h_{\left(  1\right)  }{R_{1}}^{1}m\right)
\otimes_{A}\left(  a_{\left(  1\right)  }^{\ast}\vartriangleleft{R_{2}}%
^{2}\otimes h_{\left(  2\right)  }n\right) \\
& =\left\langle b^{\ast},a\right\rangle \left(  a_{\left(  2\right)  }^{\ast
}\otimes{R_{2}}^{1}h_{\left(  1\right)  }m\right)  \otimes_{A}\left(
a_{\left(  1\right)  }^{\ast}\vartriangleleft R_{2}{}^{2}\otimes h_{\left(
2\right)  }n\right) \\
& =\left\langle b^{\ast},a\right\rangle \left(  a_{\left(  2\right)  }^{\ast
}\otimes{R_{2}}^{1}h_{\left(  1\right)  }m\right)  \otimes_{A}\left(
\left\langle a_{\left(  1\right)  }^{\ast}\vartriangleleft R_{2}{}^{2}%
,x^{1}\right\rangle x^{2}\rightharpoonup\alpha_{A}\otimes h_{\left(  2\right)
}n\right) \\
& =\left\langle b^{\ast},a\right\rangle \left\langle a_{\left(  1\right)
}^{\ast}\vartriangleleft R_{2}{}^{2},x^{1}\right\rangle \left(  a_{\left(
2\right)  }^{\ast}\otimes{R_{2}}^{1}h_{\left(  1\right)  }m\right)  \cdot
x^{2}\otimes_{A}\left(  \alpha_{A}\otimes h_{\left(  2\right)  }n\right)
\end{align*}%
\begin{align*}
& =\left\langle b^{\ast},a\right\rangle \left\langle a_{\left(  1\right)
}^{\ast},{R_{2}}^{2}\cdot x^{1}\right\rangle \left(  a_{\left(  2\right)
}^{\ast}\leftharpoonup\left(  {R_{1}}^{2}\cdot x^{2}\right)  \otimes{R_{1}%
}^{1}{R_{2}}^{1}h_{\left(  1\right)  }m\right)  \otimes_{A}\left(  \alpha
_{A}\otimes h_{\left(  2\right)  }n\right) \\
& =\xi_{M,N}\left(  \left(  a\otimes h\otimes a^{\ast}\right)  \cdot\left(
b^{\ast}\otimes\left(  m\otimes n\right)  \right)  \right)  .
\end{align*}
Thus, $\xi_{M,N}$ is an isomorphism of left $B$-modules, and it provides a
tensor structure for $F$. Let $c$, $\tilde{c}$ be the braidings for
$_{H}\mathcal{M}$ and $_{B}\mathcal{M}$, determined by $R$ and $R_{B}$,
respectively. One can check that
\[
\tilde{c}_{F\left(  M\right)  ,F\left(  N\right)  }\left(  \xi_{M,N}\left(
a^{\ast}\otimes m\otimes n\right)  \right)  =\xi_{M,N}\left(  a^{\ast}\otimes
c_{M,N}\left(  m\otimes n\right)  \right)  ,
\]
thus $F$ is a braided tensor equivalence.
\end{proof}
\end{proposition}

Note that the map (\ref{ThetaMap}) provides an algebra embedding
$\varphi:A\#H\rightarrow B$ given by $a\#h\mapsto S\left(  h_{\left(
1\right)  }\right)  \cdot a_{\left\langle 0\right\rangle }\otimes h_{\left(
2\right)  }\otimes a_{\left\langle 1\right\rangle }$.

\begin{proposition}
\label{Prop_AsmaH_is_subWHA}If the Drinfeld element $u$ of $\left(
H,R\right)  $ acts trivially on $A$, then $\varphi:A\#H\rightarrow B$ is a
weak Hopf algebra monomorphism.
\end{proposition}

\begin{proof}
For any $a\in A$, $h\in H$,
\begin{align*}
& \left(  \varphi\otimes\varphi\right)  \tilde{\Delta}\left(  a\#h\right) \\
& =\varphi\left(  a\left(  R_{1}{}^{2}\cdot x^{1}\right)  \otimes R_{1}{}%
^{1}h_{\left(  1\right)  }\right)  \otimes\varphi\left(  x^{2}\otimes
h_{\left(  2\right)  }\right) \\
& =S\left(  R_{1}{}^{1}h_{\left(  1\right)  }\right)  \cdot\left(
a_{\left\langle 0\right\rangle }\left(  \left(  R_{1}{}^{2}R_{2}{}^{2}\right)
\cdot{x}^{1}\right)  \right)  \otimes R_{2}{}^{1}h_{\left(  2\right)  }\otimes
a_{\left\langle 1\right\rangle }\otimes\varphi\left(  x^{2}\otimes h_{\left(
3\right)  }\right) \\
& =S\left(  h_{\left(  1\right)  }\right)  \cdot\left(  \left(  S\left(
R_{3}{}^{1}\right)  \cdot a_{\left\langle 0\right\rangle }\right)  \left(
\left(  S\left(  R_{1}{}^{1}\right)  R_{1}{}^{2}R_{3}{}^{2}R_{2}{}^{2}\right)
\cdot{x}^{1}\right)  \right)  \otimes R_{2}{}^{1}h_{\left(  2\right)  }\otimes
a_{\left\langle 1\right\rangle }\otimes\varphi\left(  x^{2}\otimes h_{\left(
3\right)  }\right) \\
& =S\left(  h_{\left(  1\right)  }\right)  \cdot\left(  \left(  R_{2}{}%
^{2}\cdot{x}^{1}\right)  a_{\left\langle 0\right\rangle }\right)  \otimes
R_{2}{}^{1}h_{\left(  2\right)  }\otimes a_{\left\langle 1\right\rangle
}\otimes\varphi\left(  x^{2}\otimes h_{\left(  3\right)  }\right) \\
& =S\left(  h_{\left(  1\right)  }\right)  \cdot\left(  \left(  R_{2}{}%
^{2}\cdot\left(  {x_{1}}^{1}{x_{2}}^{1}\right)  \right)  a_{\left\langle
0\right\rangle }\right)  \otimes R_{2}{}^{1}h_{\left(  2\right)  }\otimes
a_{\left\langle 1\right\rangle }\otimes S\left(  h_{\left(  3\right)
}\right)  \cdot x{_{1}}^{2}\otimes h_{\left(  4\right)  }\otimes{x_{2}}%
^{2}\rightharpoonup\alpha_{A}\\
& =S\left(  h_{\left(  1\right)  }\right)  \cdot\left(  \left(  \left(
R_{1}{}^{2}h_{\left(  3\right)  }\right)  \cdot{x_{1}}^{1}\right)  \left(
R_{2}{}^{2}\cdot{x_{2}}^{1}\right)  a_{\left\langle 0\right\rangle }\right)
\otimes R_{2}{}^{1}R_{1}{}^{1}h_{\left(  2\right)  }\otimes a_{\left\langle
1\right\rangle }\otimes x{_{1}}^{2}\otimes h_{\left(  4\right)  }\otimes
{x_{2}}^{2}\rightharpoonup\alpha_{A}\\
& =\left(  R_{1}{}^{2}\cdot{x_{1}}^{1}\right)  \left(  S\left(  h_{\left(
1\right)  }\right)  \cdot\left(  {x_{2}}^{1}a_{\left\langle 0\right\rangle
}\right)  \right)  \otimes R_{2}{}^{1}R_{4}{}^{1}h_{\left(  2\right)  }R_{1}%
{}^{1}\otimes a_{\left\langle 1\right\rangle }\\
& \otimes x{_{1}}^{2}\otimes h_{\left(  3\right)  }\otimes\left(  {x_{2}}%
^{2}\rightharpoonup\left(  \alpha_{A}\vartriangleleft S\left(  R_{4}{}%
^{2}\right)  \right)  \right)  \vartriangleleft R_{2}{}^{2}\\
& =\left(  R_{1}{}^{2}\cdot{x_{1}}^{1}\right)  \left(  S\left(  h_{\left(
1\right)  }\right)  \cdot\left(  {x_{2}}^{1}a_{\left\langle 0\right\rangle
}\right)  \right)  \otimes R_{2}{}^{1}h_{\left(  2\right)  }R_{1}{}^{1}\otimes
a_{\left\langle 1\right\rangle }\otimes x{_{1}}^{2}\otimes h_{\left(
3\right)  }\otimes\left(  {x_{2}}^{2}\rightharpoonup\alpha_{A}\right)
\vartriangleleft R_{2}{}^{2}\\
& =\left(  R_{1}{}^{2}\cdot{x}^{1}\right)  \left(  S\left(  h_{\left(
1\right)  }\right)  \cdot a_{\left\langle 0\right\rangle }\right)  \otimes
R_{2}{}^{1}h_{\left(  2\right)  }R_{1}{}^{1}\otimes a_{\left\langle
2\right\rangle }\otimes x^{2}\otimes h_{\left(  3\right)  }\otimes
a_{\left\langle 1\right\rangle }\vartriangleleft R_{2}{}^{2}\\
& =\Delta_{B}\left(  S\left(  h_{\left(  1\right)  }\right)  \cdot
a_{\left\langle 0\right\rangle }\otimes h_{\left(  2\right)  }\otimes
a_{\left\langle 1\right\rangle }\right)  =\Delta_{B}\left(  \varphi\left(
a\#h\right)  \right)  .
\end{align*}
This completes the proof.
\end{proof}

\begin{remark}
With the right $A\#H$-module structure on $A^{\ast}\otimes H$ defined by
\[
\left(  a^{\ast}\otimes h\right)  \left(  a\#f\right)  =a^{\ast}%
\leftharpoonup\left(  h_{\left(  1\right)  }\cdot a\right)  \otimes h_{\left(
2\right)  }f,\ \text{for any }a^{\ast}\in A^{\ast},\ a\in A,\ h,g\in H,
\]
the map $\varphi$ is the composition of maps in the following diagram.
\begin{equation*}%
\def\mleftdelim{.}\def\mrightdelim{.}\def\mrowsep{1cm}\def\mcolumnsep{2cm}%
\begin{tikzpicture}[scale=1,samples=40,baseline]\matrix (m) [matrix of math nodes,left delimiter={\mleftdelim},right delimiter={\mrightdelim},row sep=\mrowsep,column sep=\mcolumnsep]{ |[name=A]|A\#H & |[name=B]|\End\left( A^{\ast }\right) \otimes H\cong B \\ |[name=C]|\End_{A\#H}\left( A^{\ast }\otimes H\right) & |[name=D]|\End_{H}\left( A^{\ast }\otimes H\right) \\};\begin{scope}[every node/.style={midway,auto,font=\scriptsize}] \draw (A) edge[->] node[above] {$\varphi$} (B) edge[->] node[left] {$\cong$} (C) (B) edge[<-] node[right] {$\cong $} (D) (C) edge[to reversed-to] node[above] {}(D); \end{scope} 
\end{tikzpicture}
\end{equation*}

The image of $\varphi$ is actually
\[
\operatorname{Im}\varphi=\left\{  \sum_{i}a_{i}\otimes h_{i}\otimes
a_{i}^{\ast}\in B\mid\sum_{i}aa_{i}\otimes h_{i}\otimes a_{i}^{\ast}=\sum
_{i}a_{i}\otimes\left(  h_{i}\otimes a_{i}^{\ast}\right)  \blacktriangleleft
a,\forall a\in A\right\}  ,
\]
where $\left(  h\otimes a^{\ast}\right)  \blacktriangleleft a=h_{\left(
2\right)  }\otimes a^{\ast}\leftharpoonup\left(  h_{\left(  1\right)  }\cdot
a\right)  $.

\end{remark}
\begin{proposition}
Let $R_{B}$ be the R-matrix of $B$ as in Proposition~\ref{prop_wha_stru_on_B}.
Then $R_{B}\in\operatorname{Im}\varphi\otimes\operatorname{Im}\varphi$ if and
only if $A$ is in the M\"{u}ger center of {}$_{H}\mathcal{M}$.

Moreover, if $A$ is in the M\"{u}ger center of {}$_{H}\mathcal{M}$, then
$\varphi:A\#H\rightarrow B$ is a quasi-triangular weak Hopf algebra map.
\end{proposition}

\begin{proof}
First, we show that $R_{B}\subseteq B\otimes\operatorname{Im}\varphi$. Since
for any $a\in A$,%
\begin{align*}
& \left(  {R_{3}}^{2}\cdot{x_{2}}^{1}\right)  {x_{1}}^{1}\otimes{R_{2}}%
^{1}{R_{3}}^{1}\otimes\alpha_{A\left(  1\right)  }\vartriangleleft{R_{1}}%
^{2}\otimes{x_{2}}^{2}\otimes\left(  {R_{1}}^{1}{R_{2}}^{2}\otimes x{_{1}}%
^{2}\rightharpoonup\alpha_{A\left(  2\right)  }\right)  \blacktriangleleft a\\
& =\left(  {R_{3}}^{2}\cdot{x_{2}}^{1}\right)  {x_{1}}^{1}\otimes{R_{2}}%
^{1}{R_{5}}^{1}{R_{3}}^{1}\otimes\alpha_{A\left(  1\right)  }\vartriangleleft
\left(  {R_{1}}^{2}{R_{4}}^{2}\right) \\
&  \otimes{x_{2}}^{2}\otimes{R_{4}}^{1}{R_{2}}^{2}\otimes x{_{1}}^{2}
\rightharpoonup\alpha_{A\left(  2\right)  }\leftharpoonup\left(  \left(
{R_{1}}^{1}{R_{5}}^{2}\right)  \cdot a\right) \\
& =\left(  {R_{3}}^{2}\cdot{x_{2}}^{1}\right)  {x_{1}}^{1}\otimes{R_{2}}%
^{1}{R_{5}}^{1}{R_{3}}^{1}\otimes\alpha_{A\left(  1\right)  }\vartriangleleft
\left(  {R_{1}}^{2}{R_{4}}^{2}\right) \\
& \otimes{x_{2}}^{2}\otimes{R_{4}}^{1}{R_{2}}^{2}\otimes\left\langle
\alpha_{A\left(  2\right)  }\vartriangleleft{R_{1}}^{1},{R_{5}}^{2}\cdot
a\right\rangle x{_{1}}^{2}\rightharpoonup\alpha_{A\left(  3\right)  }\\
& =\left(  {R_{3}}^{2}\cdot{x_{2}}^{1}\right)  {x_{1}}^{1}\otimes{R_{2}}%
^{1}{R_{5}}^{1}{R_{3}}^{1}\otimes\alpha_{A\left(  3\right)  }\vartriangleleft
{R_{4}}^{2}\\
& \otimes{x_{2}}^{2}\otimes{R_{4}}^{1}{R_{2}}^{2}\otimes\left\langle
\alpha_{A\left(  2\right)  },{R_{5}}^{2}\cdot a\right\rangle x{_{1}}%
^{2}\rightharpoonup\alpha_{A\left(  1\right)  }\\
& ={x_{1}}^{1}\otimes{R_{2}}^{1}{R_{5}}^{1}{R_{3}}^{1}\otimes\alpha_{A\left(
1\right)  }\vartriangleleft{R_{4}}^{2}\\
& \otimes{x_{2}}^{2}\otimes{R_{4}}^{1}{R_{2}}^{2}\otimes x{_{1}}^{2}\left(
{R_{3}}^{2}\cdot{x_{2}}^{1}\right)  \left(  {R_{5}}^{2}\cdot a\right)
\rightharpoonup\alpha_{A\left(  2\right)  }\\
& ={x_{1}}^{1}\otimes{R_{2}}^{1}{R_{3}}^{1}\otimes\alpha_{A\left(  1\right)
}\vartriangleleft{R_{4}}^{2}\otimes{x_{2}}^{2}\otimes{R_{4}}^{1}{R_{2}}%
^{2}\otimes x{_{1}}^{2}\left(  {R_{3}}^{2}\cdot\left(  {x_{2}}^{1}a\right)
\right)  \rightharpoonup\alpha_{A\left(  2\right)  }\\
& =\left(  {R_{3}}^{2}\cdot{x_{2}}^{1}\right)  {x_{1}}^{1}\otimes{R_{2}}%
^{1}{R_{3}}^{1}\otimes\alpha_{A\left(  1\right)  }\vartriangleleft{R_{1}}%
^{2}\otimes a{x_{2}}^{2}\otimes{R_{1}}^{1}{R_{2}}^{2}\otimes x{_{1}}%
^{2}\rightharpoonup\alpha_{A\left(  2\right)  }.
\end{align*}
If for any $a\in A$, $\left(  {R_{1}}^{2}{R_{2}}^{1}\right)  \cdot
a\otimes{R_{1}}^{1}{R_{2}}^{2}=a\otimes1_{H}$, then
\begin{align*}
& \left(  {R_{3}}^{2}\cdot{x_{2}}^{1}\right)  {x_{1}}^{1}\otimes\left(
{R_{2}}^{1}{R_{3}}^{1}\otimes\alpha_{A\left(  1\right)  }\vartriangleleft
{R_{1}}^{2}\right)  \blacktriangleleft a\otimes{x_{2}}^{2}\otimes{R_{1}}%
^{1}{R_{2}}^{2}\otimes x{_{1}}^{2}\rightharpoonup\alpha_{A\left(  2\right)
}\\
& =\left(  {R_{3}}^{2}\cdot{x_{2}}^{1}\right)  {x_{1}}^{1}\otimes\left(
{R_{2}}^{1}{R_{3}}^{1}\right) _{\left(  2\right)  }\otimes\left(
\alpha_{A\left(  1\right)  }\vartriangleleft{R_{1}}^{2}\right)  \leftharpoonup
\left(  \left(  {R_{2}}^{1}{R_{3}}^{1}\right) _{\left(  1\right)  }\cdot
a\right) \\
& \otimes{x_{2}}^{2}\otimes{R_{1}}^{1}{R_{2}}^{2}\otimes x{_{1}}%
^{2}\rightharpoonup\alpha_{A\left(  2\right)  }%
\end{align*}
\begin{align*}
& =\left(  \left(  {R_{3}}^{2}{R_{5}}^{2}\right)  \cdot{x_{2}}^{1}\right)
{x_{1}}^{1}\otimes\left(  {R_{4}}^{1}{R_{5}}^{1}\otimes\left(  \alpha
_{A\left(  1\right)  }\vartriangleleft{R_{1}}^{2}\right)  \leftharpoonup
\left(  \left(  {R_{2}}^{1}{R_{3}}^{1}\right)  \cdot a\right)  \right) \\
& \otimes{x_{2}}^{2}\otimes{R_{1}}^{1}{R_{2}}^{2}{R_{4}}^{2}\otimes x{_{1}%
}^{2}\rightharpoonup\alpha_{A\left(  2\right)  }\\
& =\left(  \left(  {R_{3}}^{2}{R_{5}}^{2}\right)  \cdot{x_{2}}^{1}\right)
{x_{1}}^{1}\otimes{R_{4}}^{1}{R_{5}}^{1}\otimes\left\langle \alpha_{A\left(
3\right)  },\left(  \left(  {R_{6}}^{2}{R_{2}}^{1}{R_{3}}^{1}\right)  \cdot
a\right)  \right\rangle \alpha_{A\left(  1\right)  }\vartriangleleft{R_{1}%
}^{2}\\
& \otimes{x_{2}}^{2}\otimes{R_{1}}^{1}{R_{6}}^{1}{R_{2}}^{2}{R_{4}}^{2}\otimes
x{_{1}}^{2}\rightharpoonup\alpha_{A\left(  2\right)  }\\
& =a\left(  {R_{5}}^{2}\cdot{x_{2}}^{1}\right)  {x_{1}}^{1}\otimes{R_{4}}%
^{1}{R_{5}}^{1}\otimes\alpha_{A\left(  1\right)  }\vartriangleleft{R_{1}}%
^{2}\otimes{x_{2}}^{2}\otimes{R_{1}}^{1}{R_{4}}^{2}\otimes x{_{1}}%
^{2}\rightharpoonup\alpha_{A\left(  2\right)  },
\end{align*}
so $R_{B}\in\operatorname{Im}\varphi\otimes B$, and thus $R_{B}\in
\operatorname{Im}\varphi\otimes\operatorname{Im}\varphi$.

Conversely, suppose that $R_{B}\in\operatorname{Im}\varphi\otimes
\operatorname{Im}\varphi$, then for any $a\in A$,
\begin{align*}
& \left\langle \alpha_{A\left(  3\right)  },\left(  \left(  {R_{6}}^{2}{R_{2}%
}^{1}{R_{3}}^{1}\right)  \cdot a\right)  \right\rangle \left(  \left(  {R_{3}%
}^{2}{R_{5}}^{2}\right)  \cdot{x_{2}}^{1}\right)  {x_{1}}^{1}\otimes{R_{4}%
}^{1}{R_{5}}^{1}\otimes\alpha_{A\left(  1\right)  }\vartriangleleft{R_{1}}%
^{2}\\
& \otimes{x_{2}}^{2}\otimes{R_{1}}^{1}{R_{6}}^{1}{R_{2}}^{2}{R_{4}}^{2}\otimes
x{_{1}}^{2}\rightharpoonup\alpha_{A\left(  2\right)  }\\
& =a\left(  {R_{5}}^{2}\cdot{x_{2}}^{1}\right)  {x_{1}}^{1}\otimes{R_{4}}%
^{1}{R_{5}}^{1}\otimes\alpha_{A\left(  1\right)  }\vartriangleleft{R_{1}}%
^{2}\otimes{x_{2}}^{2}\otimes{R_{1}}^{1}{R_{4}}^{2}\otimes x{_{1}}%
^{2}\rightharpoonup\alpha_{A\left(  2\right)  }.
\end{align*}
Applying $id\otimes\varepsilon\otimes1_{A}\otimes\alpha_{A}\otimes
id\otimes1_{A}$ to both sides, we obtain $\left(  {R_{1}}^{2}{R_{2}}%
^{1}\right)  \cdot a\otimes{R_{1}}^{1}{R_{2}}^{2}=a\otimes1_{H}$.
\end{proof}

Let $\mathcal{C}$ be a finite tensor category and let $\mathcal{Z}\left(
\mathcal{C}\right)  $ be its Drinfeld center. Then the forgetful functor
$F_{\mathcal{C}}:\mathcal{Z}\left(  \mathcal{C}\right)  \rightarrow
\mathcal{C}$ has a right adjoint $I_{\mathcal{C}}:\mathcal{C}\rightarrow
\mathcal{Z}\left(  \mathcal{C}\right)  $.
From~\cite{EtingofN-O-2011Weakly,DavydovMugerNikshychOstrik2013Witt}, the
object $A=I_{\mathcal{C}}\left(  1\right)  $ has a natural structure of
commutative algebra in $\mathcal{Z}\left(  \mathcal{C}\right)  $, and the
functor $I_{\mathcal{C}}$ induces a tensor equivalence $\mathcal{C}\approx
A$-$\operatorname{mod}_{\mathcal{Z}\left(  \mathcal{C}\right)  }$.

We close this section by considering what happens when $\mathcal{C}={}%
_{H}\mathcal{M}$ for a finite dimensional Hopf algebra $H$. Note that the
Drinfeld double $D\left(  H\right)  $ for any finite dimensional Hopf algebra
$H$ is quasi-triangular with $R=\sum_{i=1}^{n}\varepsilon\bowtie x_{i}\otimes
p_{i}\bowtie1_{H},$ where $\left\{  x_{i}\right\} _{i=1}^{n},\ \left\{
p_{i}\right\} _{i=1}^{n}$ are dual bases of $H$ and $H^{\ast}$. Then $H$
itself is a quantum commutative left $D\left(  H\right)  $-module algebra with
actions given by%
\[
\left(  \varepsilon\bowtie h\right)  \cdot l=h_{\left(  1\right)  }lS\left(
h_{\left(  2\right)  }\right)  ,\quad\left(  p\bowtie1_{H}\right)  \cdot
l=l\leftharpoonup S^{-1}\left(  p\right)  ,\text{ for }h,l\in H,p\in H^{\ast}.
\]
We will achieve the equivalences $\mathcal{C}\approx A$-$\operatorname{mod}%
_{\mathcal{Z}\left(  \mathcal{C}\right)  }$ via an algebra isomorphism
$H\#D\left(  H\right)  \cong\mathcal{H}\left(  H^{cop}\right)  \otimes H$,
where $\mathcal{H}\left(  H^{cop}\right)  $ is the Heisenberg double of
$H^{cop}$.

\begin{proposition}
Let $H$ be a finite dimensional Hopf algebra. Then as algebras $H\#D\left(
H\right)  \cong\mathcal{H}\left(  H^{cop}\right)  \otimes H$. Moreover, the
tensor categories $_{H\#D\left(  H\right)  }\mathcal{M}$ and $_{H^{cop}%
}\mathcal{M}$ are equivalent.
\end{proposition}

\begin{proof}
Observe that $H\#\left(  H^{\ast}\bowtie1_{H}\right)  $ is a subalgebra of
$H\#D\left(  H\right)  =H\#\left(  H^{\ast}\bowtie H\right)  $, and it is
isomorphic to the usual Heisenberg double $\mathcal{H}\left(  H^{cop}\right)
=H^{cop}\#\left(  H^{cop}\right) ^{\ast}$ via $l\#p\bowtie1\mapsto
l\#S^{-1}\left(  p\right)  $, hence $H\#\left(  H^{\ast}\bowtie1_{H}\right)  $
is a central simple subalgebra of $H\#D\left(  H\right)  $.

Let $\left\{  x_{i},p_{i}\right\} _{i=1}^{n}$ be a dual basis for $H$. It is
routine to check that the subspace $C=\left\{  \sum_{i=1}^{n}S\left(
h_{\left(  1\right)  }\right)  S\left(  x_{i\left(  2\right)  }\right)
h_{\left(  3\right)  }S^{2}\left(  x_{i\left(  1\right)  }\right)  \#\left(
p_{i}\bowtie h_{\left(  2\right)  }\right)  \mid h\in H\right\}  $ is
contained in the centralizer $C_{H\#D\left(  H\right)  }\left(  H\#\left(
H^{\ast}\bowtie1_{H}\right)  \right)  $. It is well-known that $H\#D\left(
H\right)  \cong\left(  H\#H^{\ast}\right)  \otimes C_{H\#D\left(  H\right)
}\left(  H\#H^{\ast}\right)  $. Evidently $\varepsilon_{H}\otimes
\varepsilon_{H^{\ast}}\otimes id_{H}:C\rightarrow H$ is surjective, so $\dim
C\geq\dim H$. Hence $C=C_{H\#D\left(  H\right)  }\left(  H\#H^{\ast}\right)  $.

Moreover, for $h,g\in H$,
\begin{align*}
& \left(  \sum_{i=1}^{n}S\left(  x_{i\left(  2\right)  }h_{\left(  1\right)
}\right)  h_{\left(  3\right)  }S^{2}\left(  x_{i\left(  1\right)  }\right)
\#\left(  p_{i}\bowtie h_{\left(  2\right)  }\right)  \right)  \left(
\sum_{j=1}^{n}S\left(  x_{j\left(  2\right)  }g_{\left(  1\right)  }\right)
g_{\left(  3\right)  }S^{2}\left(  x_{j\left(  1\right)  }\right)  \#\left(
p_{j}\bowtie g_{\left(  2\right)  }\right)  \right) \\
& =\sum_{i=1}^{n}S\left(  \left(  hg\right) _{\left(  1\right)  }\right)
S\left(  x_{i\left(  2\right)  }\right)  \left(  hg\right) _{\left(  3\right)
}S^{2}\left(  x_{i\left(  1\right)  }\right)  \#\left(  p_{i}\bowtie\left(
hg\right) _{\left(  2\right)  }\right)  .
\end{align*}
This shows that as an algebra, $C$ is isomorphic to $H$. Hence
\[
H\#D\left(  H\right)  \cong\mathcal{H}\left(  H^{cop}\right)  \otimes H
\]
as algebras.

This isomorphism induces an equivalence of categories between $\mathcal{C}%
={}_{H}\mathcal{M}$ and $\mathcal{D}={}_{H\#D\left(  H\right)  }\mathcal{M}$.
Let $M$ be a left $H$-module. Then $H\otimes M$ is a left $H\#D\left(
H\right)  $-module via
\[
\left(  l\#\left(  p\bowtie h\right)  \right)  \cdot\left(  h^{\prime}\otimes
m\right)  =l\left(  \left(  h_{\left(  1\right)  }h^{\prime}S\left(
h_{\left(  3\right)  }\right)  \right)  \leftharpoonup S^{-1}\left(  p\right)
\right)  \otimes h_{\left(  2\right)  }m.
\]
Conversely, every left $H\#D\left(  H\right)  $-module is of this form. Then
$F:\mathcal{C\rightarrow D}$, $M\mapsto H\otimes M$ is an equivalence of
$k$-linear categories. By (\ref{MonoidalStruOnAsmashH}), $H\otimes M$ is an
$H$-bimodule with actions given by%
\[
h\cdot\left(  h^{\prime}\otimes m\right)  =hh^{\prime}\otimes m,\text{
}\left(  h^{\prime}\otimes m\right)  \cdot h=h^{\prime}h_{\left(  1\right)
}\otimes S^{-1}\left(  h_{\left(  2\right)  }\right)  m,
\]
for any $M\in\mathcal{C}$, $m\in M$, $h,h^{\prime}\in H$. Let $M,N\in
{}\mathcal{C}$, then $F\left(  M\right)  \otimes_{\mathcal{D}}F\left(
N\right)  =\left(  H\otimes M\right)  \otimes_{H}\left(  H\otimes N\right)  $
with $H\#D\left(  H\right)  $-action given by
\begin{align*}
& \left(  l\#\left(  p\bowtie h\right)  \right)  \cdot\left(  \left(
h^{\prime}\otimes m\right)  \otimes_{H}\left(  1_{H}\otimes n\right)  \right)
\\
& =\left(  \left(  l\#\left(  p_{\left(  2\right)  }\bowtie h_{\left(
1\right)  }\right)  \right)  \cdot\left(  h^{\prime}\otimes m\right)  \right)
\otimes_{H}\left(  \left(  1\#\left(  p_{\left(  1\right)  }\bowtie h_{\left(
2\right)  }\right)  \right)  \cdot\left(  1_{H}\otimes n\right)  \right) \\
& =\left\langle S^{-1}\left(  p\right)  ,h_{\left(  1\right)  }h_{\left(
1\right)  }^{\prime}S\left(  h_{\left(  8\right)  }\right)  \right\rangle
\left(  lh_{\left(  2\right)  }h_{\left(  2\right)  }^{\prime}S\left(
h_{\left(  4\right)  }\right)  \otimes h_{\left(  3\right)  }m\right)
\cdot\left(  h_{\left(  5\right)  }S\left(  h_{\left(  7\right)  }\right)
\right)  \otimes_{H}\left(  1_{H}\otimes h_{\left(  6\right)  }n\right) \\
& =\left(  l\left(  \left(  h_{\left(  1\right)  }h^{\prime}S\left(
h_{\left(  4\right)  }\right)  \right)  \leftharpoonup S^{-1}\left(  p\right)
\right)  \otimes h_{\left(  3\right)  }m\right)  \otimes_{H}\left(
1_{H}\otimes h_{\left(  2\right)  }n\right)  .
\end{align*}
So we see that a tensor structure on $F:\mathcal{C}^{op}\rightarrow
\mathcal{D}$ is given by
\[
F\left(  M\right)  \otimes_{\mathcal{D}}F\left(  N\right)  =\left(  H\otimes
M\right)  \otimes_{H}\left(  H\otimes N\right)  \cong H\otimes\left(  M\otimes
N\right)  =F\left(  M\otimes_{\mathcal{C}^{op}}N\right)  ,
\]
where $\mathcal{C}^{op}={}_{H^{cop}}\mathcal{M}$ with monoidal structure
opposite to $\mathcal{C}$. Therefore, ${}_{H\#D\left(  H\right)  }\mathcal{M}$
is tensor equivalent to $_{H^{cop}}\mathcal{M}$.
\end{proof}

Let $\mathcal{C}$ be a fusion category. We denote the set of isomorphism
classes of simple objects of $\mathcal{C}$ by $\operatorname*{Irr}\left(
\mathcal{C}\right)  $. Recall from~\cite{Etingof2005On} that the
Frobenius-Perron dimension $\operatorname{FPdim}X$ of an object $X\in
\mathcal{C}$ is defined to be the largest positive eigenvalue of the fusion
matrix, that is, the matrix of left or right multiplication by the class of
$X$ with respect to the basis $\operatorname*{Irr}\left(  \mathcal{C}\right)
$ of the Grothendieck ring $\operatorname*{Gr}\left(  \mathcal{C}\right)  $ of
$\mathcal{C}$. Moreover, the number $\operatorname{FPdim}X$ is an algebraic
integer, as it is an eigenvalue of an integral matrix.

Now consider the weak Hopf algebra $A\#H$ in
Theorem~\ref{Thm_WHA_stru_of_AsmH}. If $A$ is $H$-simple, then $A$ is a simple
object of multifusion category $_{A\#H}\mathcal{M}$. Hence $_{A\#H}%
\mathcal{M}$ is a fusion category, and we can compute the Frobenius-Perron
dimensions of objects.

\begin{theorem}
\label{Thm_fpdim}Let $\left(  H,R\right)  $ be a semisimple quasi-triangular
Hopf algebra. Let $A$ be a strongly separable quantum commutative $H$-simple
module algebra such that $u$ acts on $A$ trivially. Then for any nonzero
$V\in{}_{A\#H}\mathcal{M}$, we have
\[
\dim A\mid\dim V,\text{and }\operatorname{FPdim}V=\frac{\dim V}{\dim A}.
\]

\end{theorem}

\begin{proof}
Let $L$ be an arbitrary weak Hopf algebra, and let $L_{t}$, $L_{s}$ be its
target and source subalgebras. We view $L$ as a right $L_{s}$-module via the
multiplication and as a right $L_{t}$-module via $l\cdot z=S\left(  z\right)
l $, for $l\in L$, $z\in L_{t}$. For any $V\in{}_{L}\mathcal{M}$,
$L\otimes_{L_{t}}V$ is a left $L$-module via the diagonal action, i.e.
$l^{\prime}\left(  l\otimes v\right)  =l_{\left(  1\right)  }^{\prime}l\otimes
l_{\left(  2\right)  }^{\prime}v$. Then there is an $L$-module isomorphism
\[
L\otimes_{L_{t}}V\rightarrow L\otimes_{L_{s}}V,l\otimes v\mapsto l_{\left(
1\right)  }\otimes S\left(  l_{\left(  2\right)  }\right)  v,
\]
where we view $L\otimes_{L_{s}}V$ as a left $L$-module via the left
multiplication on $L$.

Now take $L=A\#H$, so $L_{t}=A\#1_{H}\cong A$ and $L_{s}=\left\{  R^{2}\cdot
a\#R^{1}\mid a\in A\right\}  \cong A^{op}$. If $V\in{}_{A\#H}\mathcal{M}$,
$V\in\mathcal{M}_{A}^{H^{\ast}}$ with the $H^{\ast}$-coaction induced from the
left $H$-action and the right $A$-module action given by $v\cdot a=\left(
R^{2}\cdot a\#R^{1}\right)  v,$ for $a\in A$, $v\in V$. Note that $A$ is an
$H^{\ast}$-simple comodule algebra, so by Lemma~\ref{lemma_skryabin}, there
exist $m,n\in\mathbb{N}$ such that $V^{n}\cong A^{m}$ as right $A$-modules,
hence $V^{n}\cong\left(  L_{s}\right) ^{m}$ as left $L_{s}$-modules, and
$L\otimes_{L_{t}}V^{n}\cong L^{m}$ as left $L$-modules.

In the Grothendieck ring $\operatorname*{Gr}\left( _{L}\mathcal{M}\right)  $,
this implies that $[L]\cdot(n[V])=(m[L])$, hence $\left[  L\right]  $ is an
eigenvector with positive eigenvalue $\dfrac{m}{n}$ for the matrix $N_{V}$ of
right multiplication $\left[  V\right]  $ on the basis $\operatorname*{Irr}%
(_{L}\mathcal{M})$, so by Frobenius-Perron theorem (\cite[XIII.2]%
{Gantmacher1998theory}),
\[
\operatorname{FPdim}V=\frac{m}{n}=\frac{\dim V}{\dim A}%
\]
is an integer.
\end{proof}

\section{Almost-Triangularity and Applications}

Recall from~\cite{LiuZhu2007Almost} that a quasi-triangular Hopf algebra
$\left(  H,R\right)  $ is called almost-triangular if ${R}^{21}R\in Z\left(
H\right)  \otimes Z\left(  H\right)  $, or, equivalently, if ${R}^{21}R\in
Z\left(  H\right)  \otimes H$, where $Z\left(  H\right)  $ denotes the center
of $H$.

Let $\mathcal{C}$ be a finite braided multitensor category with braided
$\sigma$. Then the end $U\left(  \mathcal{C}\right)  =\int_{X\in\mathcal{C}%
}X\otimes X^{\ast}$ has a natural structure of Hopf algebra in $\mathcal{C}$
(see~\cite{Lyubashenko1995Modular,Majid1995foundations} for the definition of
$U\left(  \mathcal{C}\right)  $). If $\mathcal{C}={}_{H}\mathcal{M}$ for a
finite dimensional quasi-triangular Hopf algebra $\left(  H,R\right)  $, then
the braided Hopf algebra $U\left(  \mathcal{C}\right)  $ is precisely the
transmuted braided group $H_{R}$.

For a braided multitensor category $\left(  \mathcal{C},\sigma\right)  $, the
\textit{M\"{u}ger center} of $\mathcal{C}$ is the full subcategory
$\mathcal{C}^{\prime}$ consisting of all objects $V\in\mathcal{C}$ such that
$\sigma_{X,V}\circ\sigma_{V,X}=id_{X\otimes X}$, for all $X\in\mathcal{C}$.

A cocommutative coalgebra in a braided category $\left(  \mathcal{C}%
,\sigma\right)  $ is a coalgebra $\left(  C,\Delta_{C},\varepsilon_{C}\right)
$ in $\mathcal{C}$ such that $\Delta_{C}=\sigma_{C,C}\circ\Delta_{C}$. The
next definition can be viewed as a categorical version of almost-triangularity.

\begin{definition}
A finite braided multitensor category $\mathcal{C}$ is called almost-symmetric
if the braided Hopf algebra $U\left(  \mathcal{C}\right)  $ is cocommutative
in $\mathcal{C}$.
\end{definition}

The following proposition illustrates the meaning of this notion.

\begin{proposition}
\label{Proposition_H_R*_is_q_comm}Let $\left(  H,R\right)  $ be a finite
dimensional quasi-triangular Hopf algebra. Then the following are equivalent:

\begin{enumerate}
\item The braided tensor category $_{H}\mathcal{M}$ is almost-symmetric.

\item $\left(  H,R\right)  $ is almost-triangular.

\item The left module algebra $H_{R}{}^{\ast}$ is quantum commutative with
respect to $\left(  H^{op},R^{21}\right)  $.

\item For any $h\in H$, $\left(  {R_{2}}^{2}{R_{1}}^{1}\right)  \cdot
_{ad}h\otimes R_{2}{}^{1}R_{1}{}^{2}=h\otimes1_{H}$, that is, $\left(
H,\cdot_{ad}\right)  $ is contained in the M\"{u}ger center of $\mathcal{C}$.
\end{enumerate}
\end{proposition}

\begin{proof}
1) $\Longleftrightarrow$ 3) is obvious.

1) $\Longleftrightarrow$ 4). Note that the braided category $_{H}\mathcal{M} $
is almost-symmetric if and only if for any $h\in H$,
\begin{equation}
\Delta_{R}\left(  h\right)  =R^{2}\cdot_{ad}h^{\left(  2\right)  }\otimes
R^{1}\cdot_{ad}h^{\left(  1\right)  }.\label{HRQuantumCocomm}%
\end{equation}
The right-hand side equals to
\[
\left(  R_{2}{}^{2}{R_{3}}^{2}R_{1}{}^{1}\right)  \cdot_{ad}h_{\left(
2\right)  }\otimes R_{2}{}^{1}h_{\left(  1\right)  }S\left(  {R_{3}}^{1}%
R_{1}{}^{2}\right)  ,
\]
while the left-hand side equals to
\[
R^{2}\cdot_{ad}h_{\left(  2\right)  }\otimes{R}^{1}h_{(1)}.
\]
So (\ref{HRQuantumCocomm}) is equivalent to
\[
\left(  {R_{3}}^{2}R_{1}{}^{1}\right)  \cdot_{ad}h_{\left(  2\right)  }\otimes
h_{\left(  1\right)  }S\left(  {R_{3}}^{1}R_{1}{}^{2}\right)  =h_{\left(
2\right)  }\otimes h_{(1)},\ \forall h\in H\text{.}%
\]
It is equivalent to $\left(  {R_{2}}^{2}{R_{1}}^{1}\right)  \cdot_{ad}h\otimes
R_{2}{}^{1}R_{1}{}^{2}=h\otimes1_{H}$, for all $h\in H$.

4) $\Longrightarrow$ 2). Equality $\left(  {R_{2}}^{2}{R_{1}}^{1}\right)
\cdot_{ad}h\otimes R_{2}{}^{1}R_{1}{}^{2}=h\otimes1_{H}$ implies that
\begin{align*}
& {R_{2}}^{2}{R_{1}}^{1}hS\left(  R_{3}{}^{1}\right)  S\left(  R_{4}{}%
^{2}\right)  \otimes R_{4}{}^{1}R_{2}{}^{1}R_{1}{}^{2}R_{3}{}^{2}\\
& =hR_{2}{}^{2}R_{1}{}^{1}S\left(  R_{3}{}^{1}\right)  S\left(  R_{4}{}%
^{2}\right)  \otimes R_{4}{}^{1}R_{2}{}^{1}R_{1}{}^{2}R_{3}{}^{2}.
\end{align*}
Therefore, ${R_{2}}^{2}{R_{1}}^{1}h\otimes R_{2}{}^{1}R_{1}{}^{2}=hR_{2}{}%
^{2}R_{1}{}^{1}\otimes R_{2}{}^{1}R_{1}{}^{2}$. Thus 4) implies that
$R^{21}R\in Z\left(  H\right)  \otimes H$.

2) $\Longrightarrow$ 4). Suppose that $\left(  H,R\right)  $ is
almost-triangular. Then by~\cite[Lemma 2.1]{LiuZhu2007Almost}, we see that
$\left(  \Delta\otimes id\right)  \left(  R^{21}R\right)  \in Z\left(
H\otimes H\otimes H\right)  $, and 4) holds.
\end{proof}

In the terminology of weak Hopf algebra, we have

\begin{proposition}
\label{prop_AT_WHA}Let $\left(  H,R\right)  $ be a finite dimensional
quasi-triangular weak Hopf algebra, $H_{s}$ and $H_{t}$ be the source and
target subalgebras of $H$. Then the following are equivalent:

\begin{enumerate}
\item The braided multitensor category $_{H}\mathcal{M}$ is almost-symmetric.

\item $R^{21}R\in C_{H}\left(  C_{H}\left(  H_{s}\right)  \right)  \otimes H$.

\item $R^{21}R\in H\otimes C_{H}\left(  C_{H}\left(  H_{t}\right)  \right)  $.

\item $R^{21}R\in C_{H}\left(  C_{H}\left(  H_{s}\right)  \right)  \otimes
C_{H}\left(  C_{H}\left(  H_{t}\right)  \right)  $.

\item For any $b\in C_{H}\left(  H_{s}\right)  $, $\left(  {R_{2}}^{2}{R_{1}%
}^{1}\right)  \cdot_{ad}b\otimes R_{2}{}^{1}R_{1}{}^{2}=1_{\left(  1\right)
}\cdot_{ad}b\otimes1_{\left(  2\right)  }$, that is, $\left(  C_{H}\left(
H_{s}\right)  ,\cdot_{ad}\right)  $ is contained in the M\"{u}ger center of
$_{H}\mathcal{M}$.

\item $R^{21}R\in Z\left(  \Delta\left(  1\right)  \left(  H\otimes H\right)
\Delta\left(  1\right)  \right)  $.
\end{enumerate}
\end{proposition}

\begin{proof}
1) $\Longleftrightarrow$ 2) $\Longleftrightarrow$ 5) can be verified exactly
in the same way as the Hopf case.

Clearly 4) $\Longleftrightarrow$ 2) and 3). We show 2) $\Longrightarrow$ 3).
Suppose that $R^{21}R\in C_{H}\left(  C_{H}\left(  H_{s}\right)  \right)
\otimes H$, then for any $b\in C_{H}\left(  H_{s}\right)  $,
\begin{align*}
R_{2}{}^{2}R_{1}{}^{1}\otimes R_{2}{}^{1}R_{1}{}^{2}S\left(  b\right)   &
=S\left(  R_{1}{}^{1}R_{2}{}^{2}\right)  \otimes S\left(  bR_{1}{}^{2}R_{2}%
{}^{1}\right) \\
& =S\left(  R_{1}{}^{1}R_{2}{}^{2}\right)  \otimes S\left(  R_{1}{}^{2}R_{2}%
{}^{1}b\right) \\
& =R_{2}{}^{2}R_{1}{}^{1}\otimes S\left(  b\right)  R_{2}{}^{1}R_{1}{}^{2}.
\end{align*}
Since $H_{t}=S\left(  H_{s}\right)  $, $C_{H}\left(  H_{t}\right)  =S\left(
C_{H}\left(  H_{s}\right)  \right)  $, thus 2) implies that 3). Similarly, 3)
$\Longrightarrow$ 2), so 1) through 5) are equivalent.

We will show 3) $\Longleftrightarrow$ 6). Assume 6) holds. Then for any $a\in
C_{H}\left(  H_{t}\right)  $,
\begin{align*}
R_{2}{}^{2}R_{1}{}^{1}\otimes R_{2}{}^{1}R_{1}{}^{2}a  & =R_{2}{}^{2}R_{1}%
{}^{1}1_{\left(  1\right)  }1_{\left(  1^{\prime}\right)  }\otimes R_{2}{}%
^{1}R_{1}{}^{2}1_{\left(  2\right)  }a1_{\left(  2^{\prime}\right)  }\\
& =1_{\left(  1\right)  }1_{\left(  1^{\prime}\right)  }R_{2}{}^{2}R_{1}{}%
^{1}\otimes1_{\left(  2\right)  }a1_{\left(  2^{\prime}\right)  }R_{2}{}%
^{1}R_{1}{}^{2}\\
& =R_{2}{}^{2}R_{1}{}^{1}\otimes aR_{2}{}^{1}R_{1}{}^{2},
\end{align*}
proving 3). Conversely, assume 3). Using the equalities
\[
R\left(  y\otimes1\right)  =R\left(  1\otimes S\left(  y\right)  \right)
,\text{ }\left(  z\otimes1\right)  R=\left(  1\otimes S\left(  z\right)
\right)  R,\text{ for }y\in H_{s},z\in H_{t},
\]
we obtain that for any $h,g\in H$,%
\begin{align*}
R^{21}R\left(  \Delta\left(  1\right)  \left(  h\otimes g\right)
\Delta\left(  1\right)  \right)   & =R_{2}{}^{2}R_{1}{}^{1}1_{\left(
1\right)  }h1_{\left(  1^{\prime}\right)  }\otimes R_{2}{}^{1}R_{1}{}%
^{2}1_{\left(  2\right)  }g1_{\left(  2^{\prime}\right)  }\\
& =h_{\left(  1\right)  }R_{2}{}^{2}R_{1}{}^{1}\otimes h_{\left(  2\right)
}R_{2}{}^{1}R_{1}{}^{2}S\left(  1_{\left(  1^{\prime}\right)  }\right)
S\left(  h_{\left(  3\right)  }\right)  g1_{\left(  2^{\prime}\right)  }\\
& =h_{\left(  1\right)  }R_{2}{}^{2}R_{1}{}^{1}\otimes h_{\left(  2\right)
}R_{2}{}^{1}R_{1}{}^{2}S\left(  1\cdot_{ad}\left(  S^{-1}\left(  g\right)
h_{\left(  3\right)  }\right)  \right) \\
& =h_{\left(  1\right)  }R_{2}{}^{2}R_{1}{}^{1}\otimes h_{\left(  2\right)
}S\left(  1\cdot_{ad}\left(  S^{-1}\left(  g\right)  h_{\left(  3\right)
}\right)  \right)  R_{2}{}^{1}R_{1}{}^{2}\\
& =h_{\left(  1\right)  }S\left(  1_{\left(  2\right)  }\right)  R_{2}{}%
^{2}R_{1}{}^{1}\otimes h_{\left(  2\right)  }S\left(  1_{\left(  1\right)
}\right)  S\left(  h_{\left(  3\right)  }\right)  gR_{2}{}^{1}R_{1}{}^{2}\\
& =1_{\left(  1\right)  }hR_{2}{}^{2}R_{1}{}^{1}\otimes1_{\left(  2\right)
}gR_{2}{}^{1}R_{1}{}^{2}\\
& =\left(  \Delta\left(  1\right)  \left(  h\otimes g\right)  \Delta\left(
1\right)  \right)  R^{21}R.
\end{align*}
Thus 6) holds.
\end{proof}

The concept of almost-triangular weak Hopf algebra was proposed in
\cite{Zhao-Liu-Wang2017Pseudotriangular}, as being a quasi-triangular weak
Hopf algebra $\left(  H,R\right)  $ such that $R^{21}R\in Z\left(
\Delta\left(  1\right)  \left(  H\otimes H\right)  \Delta\left(  1\right)
\right)  $. By Proposition \ref{prop_AT_WHA}, this coincides with the notion
of almost-symmetric category.

Let $\left(  H,R\right)  $ be a semisimple quasi-triangular Hopf algebra, and
let $\lambda$ be an integral for $H^{\ast}$ such that $\left\langle
\lambda,1_{H}\right\rangle =1$. Then it follows from~\cite{LiuZhu2019On} that
$x=R^{2}\rightharpoonup\lambda_{\left(  1\right)  }\otimes S^{\ast}\left(
\lambda_{\left(  2\right)  }\right)  \leftharpoonup\!\!\!\leftharpoonup R^{1}$
is a separability idempotent for $H_{R}^{\ast}$.

\begin{lemma}
\label{x-symmetric}The separability idempotent $x$ is symmetric.
\end{lemma}

\begin{proof}
Let $u=S\left(  R^{2}\right)  R^{1}$ be the Drinfeld element of $H$. By the
result of Drinfeld~\cite{drinfeld1990almost}, $u\left(  Su\right)
^{-1}=a^{-1}\left(  \alpha\otimes id\right)  \left(  R\right)  $, where $a\in
H$, $\alpha\in H^{\ast}$ are the distinguished group-like elements. As $H$ is
semisimple and cosemisimple, $a=1_{H}$ and $\alpha=\varepsilon$. This shows
that $u=Su $. Clearly, $u\in Z\left(  H\right)  $. Since $S(\lambda)=\lambda$
and $\lambda$ is cocommutative, for any $h,g\in H$, we have
\begin{align*}
\left\langle x^{1},h\right\rangle \left\langle x^{2},g\right\rangle  &
=\left\langle \lambda,hR^{2}S\left(  R^{1}\cdot_{ad}g\right)  \right\rangle
=\left\langle \lambda,h{R_{1}}^{2}{R_{2}}^{2}{R_{2}}^{1}S\left(  g\right)
S\left(  {R_{1}}^{1}\right)  \right\rangle \\
& =\left\langle \lambda,h{R_{1}}^{2}{u}^{-1}S\left(  g\right)  S\left(
{R_{1}}^{1}\right)  \right\rangle =\left\langle \lambda,{R_{1}}^{1}gS\left(
{u}^{-1}\right)  S\left(  {R_{1}}^{2}\right)  S\left(  h\right)  \right\rangle
\\
& =\left\langle \lambda,g{R_{1}}^{2}{u}^{-1}S\left(  h\right)  S\left(
{R_{1}}^{1}\right)  \right\rangle =\left\langle x^{1},g\right\rangle
\left\langle x^{2},h\right\rangle ,
\end{align*}
that is, $x$ is symmetric.
\end{proof}

\begin{remark}
Let $D$ be an $H$-module subcoalgebra of $H_{R}$. Then the restriction of $x$
on $D\otimes D$ yields a symmetric separability idempotent of $D^{\ast}$.
\end{remark}

Let $C\left(  H^{\ast}\right)  $ be the set of all the cocommutative elements
of $H^{\ast}$. Then $C\left(  H^{\ast}\right)  $ is a semisimple subalgebra of
$H^{\ast}$ by~\cite{Zhu1994Hopf}. Take a complete set $\left\{  F_{i}\mid
i=1,\ldots,r\right\}  $ of minimal idempotents of $C\left(  H^{\ast}\right)
$. Then $\left\{  \Lambda\leftharpoonup F_{i}H^{\ast}\mid i=1,\ldots
,s\right\}  $ are the set of non-isomorphic irreducible Yetter-Drinfeld
submodules of $H\in{}_{H}^{H}\mathcal{YD}$ by \cite{Cohen1994Supersymmetry}.
Hence $s=r$, and we may assume that $D_{i}=\Lambda\leftharpoonup F_{i}H^{\ast
}$, $i=1,\ldots,r$.

To make difference, we use $\rightharpoonup_{R}$ to denote the left transpose
of the right multiplication of $H_{R}^{\ast}$, i.e., for $h,\ h^{\prime}\in
H$, $f,f^{\prime}\in H^{\ast}$, $\left\langle f,f^{\prime}\rightharpoonup
_{R}h\right\rangle =\left\langle f\ast_{R}f^{\prime},h\right\rangle $.

\begin{proposition}
For each $1\leq i\leq r$, $F_{i}$ is a central idempotent of $H_{R}{}^{\ast}
$, and $F_{i}\rightharpoonup_{R}H_{R}=\Lambda\leftharpoonup F_{i}H^{\ast} $.
\end{proposition}

\begin{proof}
Since for all $h\in H$,%
\begin{align*}
\Delta_{R}\left(  h\right)   & =h_{\left(  1\right)  }S\left(  R^{2}\right)
\otimes{R}^{1}{}_{\left(  1\right)  }h_{\left(  2\right)  }S\left(  {R}^{1}%
{}_{\left(  2\right)  }\right) \\
& ={R}^{2}{}_{\left(  1\right)  }h_{(2)}S\left(  {R}^{2}{}_{\left(  2\right)
}\right)  \otimes{R}^{1}h_{(1)},
\end{align*}
we have for any $f,f^{\prime}\in H^{\ast}$ that
\begin{align*}
f\ast_{R}f^{\prime}  & =\left(  S\left(  {R}^{2}\right)  \rightharpoonup
f\right)  \ast\left(  S\left(  {R}^{1}{}_{\left(  2\right)  }\right)
\rightharpoonup f^{\prime}\leftharpoonup{R}^{1}{}_{\left(  1\right)  }\right)
\\
& =\left(  f^{\prime}\leftharpoonup{R}^{1}\right)  \ast\left(  S\left(
{R}^{2}{}_{\left(  2\right)  }\right)  \rightharpoonup f\leftharpoonup{R}%
^{2}{}_{\left(  1\right)  }\right)  .
\end{align*}
Note that $F_{i}\in C\left(  H^{\ast}\right)  $, hence for all $h\in H$,
$S\left(  h_{\left(  2\right)  }\right)  \rightharpoonup F_{i}\leftharpoonup
h_{\left(  1\right)  }=\varepsilon\left(  h\right)  F_{i}.$ Thus for any $f\in
H^{\ast}$, $f\ast_{R}F_{i}=f\ast F_{i}=F_{i}\ast_{R}f$, so $F_{i}$ is a
central idempotent of $H_{R}{}^{\ast}$, and the result follows.
\end{proof}

Now assume that $\left(  H,R\right)  $ is almost-triangular, then
$\Delta\left(  u\right)  =\left(  u\otimes u\right)  R^{21}R\in Z\left(
H\otimes H\right)  $, so the adjoint action of $u$ on $H_{R}$ is trivial. By
Proposition~\ref{Proposition_H_R*_is_q_comm}, $H_{R}^{\ast}$ is a quantum
commutative $H^{op}$-module algebra. Applying
Theorem~\ref{Thm_WHA_stru_of_AsmH} with $A=H_{R}^{\ast}$, we get

\begin{proposition}
Let $\left(  H,R\right)  $ be an almost-triangular Hopf algebra, $x$ be the
symmetric separability idempotent, and $\Lambda$ be the integral for $H$ such
that $\Lambda\rightharpoonup\lambda=\varepsilon$. Then the smash product
$H_{R}^{\ast}\#H^{op}$ is an almost-triangular weak Hopf algebra via the
structure maps
\begin{align*}
\tilde{\Delta}\left(  f\#h\right)   & =f\ast_{R}\left(  x^{1}\leftharpoonup
\!\!\!\leftharpoonup R^{1}\right)  \#h_{\left(  1\right)  }R^{2}\otimes
x^{2}\#h_{\left(  2\right)  },\\
\tilde{\varepsilon}\left(  f\#h\right)   & =\left\langle f,\Lambda
\right\rangle \varepsilon\left(  h\right)  ,\\
\tilde{S}\left(  f\#h\right)   & =\left(  \varepsilon\#S\left(  h\right)
\right)  \left(  f\leftharpoonup\!\!\!\leftharpoonup R^{1}\#R^{2}\right)  ,
\end{align*}
where $h\in H$ and $f\in H^{\ast}$. Its R-matrix can be chosen as
\[
\mathcal{R}=\tilde{1}_{\left(  2\right)  }\left(  \varepsilon\#R^{2}\right)
\tilde{1}_{\left(  1^{\prime}\right)  }\otimes\tilde{1}_{\left(  1\right)
}\left(  \varepsilon\#R^{1}\right)  \tilde{1}_{\left(  2^{\prime}\right)  },
\]
where $\tilde{1}_{\left(  1\right)  }\otimes\tilde{1}_{\left(  2\right)
}=\tilde{1}_{\left(  1^{\prime}\right)  }\otimes\tilde{1}_{\left(  2^{\prime
}\right)  }=\tilde{\Delta}\left(  \varepsilon\#1_{H}\right)  $.
\end{proposition}

\begin{proof}
By Theorem~\ref{Thm_WHA_stru_of_AsmH}, $\left(  H_{R}^{\ast}\#H^{op}%
,\mathcal{R}\right)  $ is quasi-triangular. We must show that it is
almost-triangular. By (\ref{SimplifyR}),
\[
\mathcal{R}^{21}\mathcal{R}=\tilde{1}_{\left(  1\right)  }\left(
\varepsilon\#R_{1}{}^{1}R_{2}{}^{2}\right)  \otimes\tilde{1}_{\left(
2\right)  }\left(  \varepsilon\#R_{1}{}^{2}R_{2}{}^{1}\right)  .
\]
We claim that $\left(  \varepsilon\#R_{1}{}^{1}R_{2}{}^{2}\right)
\otimes\left(  \varepsilon\#R_{1}{}^{2}R_{2}{}^{1}\right)  $ is in the center
of $\left(  H_{R}^{\ast}\#H^{op}\right)  \otimes\left(  H_{R}^{\ast}%
\#H^{op}\right)  $. For any $h\in H$, $f\in H^{\ast}$,
\begin{align*}
R_{1}{}^{1}R_{2}{}^{2}\otimes\left(  \varepsilon\#R_{1}{}^{2}R_{2}{}%
^{1}\right)  \left(  f\#h\right)   & =R_{1}{}^{1}R_{3}{}^{1}R_{4}{}^{2}R_{2}%
{}^{2}\otimes f\leftharpoonup\!\!\!\leftharpoonup\left(  R_{3}{}^{2}R_{4}%
{}^{1}\right)  \#R_{1}{}^{2}R_{2}{}^{1}h\\
& =R_{1}{}^{1}R_{2}{}^{2}\otimes f\#R_{1}{}^{2}R_{2}{}^{1}h\\
& =R_{1}{}^{1}R_{2}{}^{2}\otimes f\#hR_{1}{}^{2}R_{2}{}^{1}\\
& =R_{1}{}^{1}R_{2}{}^{2}\otimes\left(  f\#h\right)  \left(  \varepsilon
\#R_{1}{}^{2}R_{2}{}^{1}\right)  .
\end{align*}
This proves the claim since $R^{21}R$ is symmetric by~\cite{LiuZhu2007Almost}.
It now follows that $\left(  H_{R}^{\ast}\#H^{op},\mathcal{R}\right)  $ is almost-triangular.
\end{proof}

Let $D$ be an $H$-module subcoalgebra of $H_{R}$. Note that by
Proposition~\ref{Proposition_H_R*_is_q_comm} $D^{\ast}$ is a quantum
commutative left $H^{op}$-module algebra. Applying
Theorem~\ref{Thm_WHA_stru_of_AsmH} again, we get that $D^{\ast}\#H^{op}$ is an
almost-triangular weak Hopf algebra.

\begin{proposition}
\label{Cor_ND_is_WHA} Let $D$ be an $H$-module subcoalgebra of $H_{R}$. Then
the $R$-adjoint-stable algebra $N_{D}$ is an almost-triangular weak Hopf
algebra, and the category $_{H}^{D}\mathcal{M}$ is an almost-symmetric
multi-fusion category. In particular, if $D$ is a minimal $H$-module
subcoalgebra of $H_{R}$, then $_{H}^{D}\mathcal{M}$ is fusion.
\end{proposition}

When $\left(  H,R\right)  =\left(  kG,1\otimes1\right)  $ is a finite group
algebra, then any simple subcomodule $W$ of $H_{R}$ has the form $W=kg$, where
$g\in G$. In this case, $N_{W}=kC\left(  g\right)  $ is the group algebra of
the centralizer subgroup of $g$, and the inclusion $N_{W}\hookrightarrow H$ is
a quasi-triangular Hopf algebra embedding.

For general almost-triangular Hopf algebras, we have a similar result, which
is an application of Proposition \ref{prop_wha_stru_on_B} and
Proposition\nolinebreak ~\ref{Prop_AsmaH_is_subWHA}.

\begin{corollary}
Let $\left(  H,R\right)  $ be a semisimple and cosemisimple almost-triangular
Hopf algebra over a field of characteristic $0$. If $D$ is an $H$-module
subcoalgebra of $H_{R}$, then $N_{D}$ can be embedded into $\operatorname{End}%
D\otimes H^{op}$, as an almost-triangular weak Hopf algebra.
\end{corollary}

As a consequence of Theorem~\ref{Thm_fpdim}, we have

\begin{corollary}
Let $V$ be an irreducible Yetter-Drinfeld module in ${}_{H}^{H}\mathcal{YD}$.
Then the dimension of $D_{V}$ divides the dimension of $V$.
\end{corollary}


\begin{thebibliography}{10}

\bibitem{Cohen1994Supersymmetry}
M.~Cohen and S.~Westreich.
\newblock From supersymmetry to quantum commutativity.
\newblock {\em J. Algebra}, 168(1):1--27, 1994.

\bibitem{DavydovMugerNikshychOstrik2013Witt}
A.~Davydov, M.~M\"{u}ger, D.~Nikshych, and V.~Ostrik.
\newblock The {W}itt group of non-degenerate braided fusion categories.
\newblock {\em J. Reine Angew. Math.}, 677:135--177, 2013.

\bibitem{drinfeld1990almost}
V.~G. Drinfeld.
\newblock Almost cocommutative {H}opf algebras.
\newblock {\em Algebra i Analiz}, 1(2):30--46, 1989.

\bibitem{Etingof2005On}
P.~Etingof, D.~Nikshych, and V.~Ostrik.
\newblock On fusion categories.
\newblock {\em Ann. of Math. (2)}, 162(2):581--642, 2005.

\bibitem{EtingofN-O-2011Weakly}
P.~Etingof, D.~Nikshych, and V.~Ostrik.
\newblock Weakly group-theoretical and solvable fusion categories.
\newblock {\em Adv. Math.}, 226(1):176--205, 2011.

\bibitem{Gantmacher1998theory}
F.~R. Gantmacher.
\newblock {\em The theory of matrices. {V}ol. 1}.
\newblock AMS Chelsea Publishing, Providence, RI, 1998.
\newblock Translated from the Russian by K. A. Hirsch, Reprint of the 1959
  translation.

\bibitem{Larson1988Semisimple}
R.~G. Larson and D.~E. Radford.
\newblock Semisimple cosemisimple {H}opf algebras.
\newblock {\em Amer. J. Math.}, 110(1):187--195, 1988.

\bibitem{LiuZhu2007Almost}
G.~Liu and S.~Zhu.
\newblock Almost-triangular {H}opf algebras.
\newblock {\em Algebr. Represent. Theory}, 10(6):555--564, 2007.

\bibitem{LiuZhu2019On}
Z.~Liu and S.~Zhu.
\newblock On the structure of irreducible {Y}etter-{D}rinfeld modules over
  quasi-triangular {H}opf algebras.
\newblock {\em J. Algebra}, 539:339--365, 2019.

\bibitem{LiuZhu2022Factorizable}
Z.~Liu and S.~Zhu.
\newblock Structures of adjoint-stable algebras over factorizable {H}opf
  algebras.
\newblock {\em arXiv:2208.09670}, 2022.

\bibitem{Lyubashenko1995Modular}
V.~Lyubashenko.
\newblock Modular transformations for tensor categories.
\newblock {\em J. Pure Appl. Algebra}, 98(3):279--327, 1995.

\bibitem{Majid1991Braided}
S.~Majid.
\newblock Braided groups and algebraic quantum field theories.
\newblock {\em Lett. Math. Phys.}, 22(3):167--175, 1991.

\bibitem{Majid1995foundations}
S.~Majid.
\newblock {\em Foundations of quantum group theory}.
\newblock Cambridge University Press, Cambridge, 1995.

\bibitem{Renault1980Groupoid}
J.~Renault.
\newblock {\em A groupoid approach to {$C^{\ast} $}-algebras}, volume 793 of
  {\em Lecture Notes in Mathematics}.
\newblock Springer, Berlin, 1980.

\bibitem{skryabin2007projectivity}
S.~Skryabin.
\newblock Projectivity and freeness over comodule algebras.
\newblock {\em Trans. Amer. Math. Soc.}, 359(6):2597--2623, 2007.

\bibitem{Zhao-Liu-Wang2017Pseudotriangular}
X.~Zhao, G.~Liu, and S.~Wang.
\newblock Pseudotriangular weak {H}opf algebras.
\newblock {\em J. Algebra Appl.}, 16(7):1750137, 18, 2017.

\bibitem{Zhu2015Braided}
H.~Zhu and Y.~Zhang.
\newblock Braided autoequivalences and quantum commutative bi-{G}alois objects.
\newblock {\em J. Pure Appl. Algebra}, 219(9):4144--4167, 2015.

\bibitem{Zhu1994Hopf}
Y.~Zhu.
\newblock Hopf algebras of prime dimension.
\newblock {\em Internat. Math. Res. Notices}, 1:53--59, 1994.

\end{thebibliography}
\end{document}